\documentclass[10pt, letterpaper, twoside]{article}

\usepackage{amsmath, amsthm,amssymb,bbold,mathrsfs} 
\usepackage{amsfonts}
\usepackage{graphicx}
\usepackage[titletoc,title]{appendix}
\usepackage{enumitem}
\setlist[itemize]{itemsep=0pt,parsep=2pt,topsep=2pt}
\setlist[enumerate]{itemsep=0pt,parsep=2pt,topsep=2pt}
\usepackage[numbers]{natbib}
\usepackage[hidelinks]{hyperref}

\newtheorem{theorem}{Theorem}[section]
\newtheorem{cor}[theorem]{Corollary}
\newtheorem{prop}[theorem]{Proposition}
\newtheorem{lemma}[theorem]{Lemma}
\newtheorem{assumption}[theorem]{Assumption}
\newtheorem{definition}[theorem]{Definition}
\theoremstyle{remark}
\newtheorem{rem}[theorem]{Remark}
\theoremstyle{definition}
\newtheorem{example}{Example}[section]
\theoremstyle{plain}

\numberwithin{equation}{section}

 \def\A{\mathbb{A}}
 \def\X{\mathbb{X}}
 \def\E{\mathbb{E}}
 
 \def\R{\mathbb{R}}
 \def\An{\mathcal{A}}
 
 \def\F{\mathcal{F}}
 \def\P{\mathcal{P}}
 \def\B{\mathcal{B}}
 \def\M{\mathcal{M}}
 \def\U{\mathcal{U}}
 \def\Q{\mathcal{Q}}

 \def\uh{\underline{h}}
 \def\bh{\bar{h}}

 \def\epi{\mathop{\text{epi}}}

\def\cl{\mathop{\text{cl}}}

\def\0{\mathbf{0}}
\def\1{\mathbf{1}}
\def\ind{\mathbb{1}}
 
\def\uv{\underline{v}}
\def\um{\underline{m}}
\def\uc{\underline{c}}
\def\elims{\underline{\mathop{\text{\rm e-lim}}}}
\def\supp{\mathop{\text{\rm supp}}}

\def\myqed{\hfill \qed}

\begin{document} 
\markboth{ACOI for Borel-Space MDPs}{}

\title{Average Cost Optimality Inequality for\\ Markov Decision Processes with Borel Spaces\\ and Universally Measurable Policies\thanks{This research was funded by DeepMind, AMII and Alberta Innovates---Technology Futures (AITF).}
\thanks{This paper consists of the main results originally given in the author's arXiv eprint \cite[Section 3]{Yu19} and newly added sections for an extended discussion and illustrative examples.}}

\author{Huizhen Yu\thanks{RLAI Lab, Department of Computing Science, University of Alberta, Canada (\texttt{janey.hzyu@gmail.com})}}
\date{}

\maketitle

\begin{abstract}
We consider average-cost Markov decision processes (MDPs) with Borel state and action spaces and universally measurable policies. For the nonnegative cost model and an unbounded cost model with a Lyapunov-type stability character, we introduce a set of new conditions under which we prove the average cost optimality inequality (ACOI) via the vanishing discount factor approach. Unlike most existing results on the ACOI, our result does not require any compactness and continuity conditions on the MDPs. Instead, the main idea is to use the almost-uniform-convergence property of a pointwise convergent sequence of measurable functions as asserted in Egoroff's theorem. Our conditions are formulated in order to exploit this property. Among others, we require that for each state, on selected subsets of actions at that state, the state transition stochastic kernel is majorized by finite measures. We combine this majorization property of the transition kernel with Egoroff's theorem to prove the ACOI.
\end{abstract}

average cost; optimality inequality; 
majorization conditions


\clearpage
\tableofcontents
\clearpage

\section{Introduction} \label{sec-1}

We consider discrete-time Markov decision processes (MDPs) with Borel state and action spaces, under the average cost criterion where the objective is to minimize the (limsup) expected long-run average cost per unit time.
Specifically, we consider the universal measurability framework, which involves lower semi-analytic (l.s.a.)~one-stage cost functions and universally measurable (u.m.)~policies. It is a mathematical formulation of MDPs developed to resolve measurability difficulties in dynamic programming on Borel spaces \cite{Blk-borel,BFO74,Shr79,ShrB78,ShrB79,Str-negative}.
An in-depth study of this theoretical framework is given in the monograph~\cite[Part~II]{bs}, and optimality properties of finite- and infinite-horizon problems with discounted and undiscounted total cost criteria have been analyzed (see e.g., \cite[Part II]{bs} and \cite{MS92,Yu-tc15,YuB-mvipi}). The average cost problem has not been thoroughly studied in this framework, however; there are only a few prior results (which we will discuss below). The primary purpose of this paper is to investigate the subject further, and our central focus will be on the average cost optimality inequality (ACOI). 

The study of ACOI was initiated by Sennott \cite{Sen89}, who proved it for countable-space MDPs; prior to \cite{Sen89}, the ACOE (average cost optimality equation) was the research focus. Cavazos-Cadena's counterexample \cite{CCa91} showed that the ACOI is more general: in his example, the ACOI has a solution and yet the ACOE does not. For Borel-space MDPs, the ACOI was first established by Sch{\"a}l~\cite{Sch93} for one-stage costs that are bounded below and under, alternatively, two types of compactness and continuity conditions. Specifically, the first (resp.~second) case requires that 
the state transition stochastic kernel is continuous with respect to (w.r.t.)~setwise convergence (resp.~weak convergence) and the one-stage cost function lower semicontinuous (l.s.c.) in the action variable (resp.~the state and action variables).
The compactness conditions on the action sets in \cite{Sch93} were weakened by Hern\'{a}ndez-Lerma~\cite{HLe91} (to inf-compactness for the first case) and by Feinberg, Kasyanov, and Zadoianchuk~\cite{FKZ12} (to $\mathbb{K}$-inf-compactness for the second case; a similar, slightly stronger condition was proposed by Costa and Dufour~\cite{CoD12} independently). Extensions of these results to unbounded one-stage costs and to ACOEs have also been studied; see the textbook accounts given in Hern\'{a}ndez-Lerma and Lasserre~\cite[Chap.~10]{HL99} and for more recent advances, see Vega-Amaya~\cite{VAm03,VAm15,VAm18}, Ja\'{s}kiewicz and Nowak \cite{JaN06}, Feinberg et al.~\cite{FKL19,FeL17} and the references therein. (For related earlier researches, see also the survey by Arapostathis et al.~\cite{ArB93}.)

In this paper, we prove the ACOI for two types of MDPs in the universal measurability framework: the nonnegative cost model and an unbounded cost model with a Lyapunov-type stability property. These models have been studied in the references just mentioned. 
As in some of those studies, to prove the ACOI, we use the vanishing discount factor approach, which treats the average cost problem as the limiting case of the discounted problems, and we adopt some boundedness conditions formulated in \cite{FKZ12,HL99,Sch93} regarding the optimal value functions of the discounted problems. 

Different from those studies, however, our results require \emph{no compactness and continuity conditions} on the one-stage cost function and state transition stochastic kernel.
Instead, we introduce a set of new conditions of, what we call, the majorization type: among others, we require that for each state, on selected subsets of actions at that state, the state transition stochastic kernel is majorized by finite measures (see Assumptions~\ref{cond-uc-2},~\ref{cond-pc-2}). Our main idea is to combine these majorization properties with Egoroff's theorem, which asserts that pointwise convergence of functions is ``almost'' uniform convergence as measured by a given finite measure (cf.\ footnote~\ref{footnote-Egroff}), and which allows us to extract arbitrarily large sets (large as measured by the majorizing finite measures) on which certain functions involved in our analyses have desired uniform convergence properties. With this technique, we obtain the ACOI for the two MDP models mentioned above (see Theorems~\ref{thm-uc-acoi},~\ref{thm-pc-acoi}). These results can be applied to a class of MDPs with discontinuous dynamics and one-stage costs.

For comparison, let us discuss several prior researches relevant to average-cost MDPs with u.m.\ policies. 
Dynkin and Yushkevich \cite[Chap.~7.9]{DyY79} and Piunovskiy~\cite{Piu89} studied the characteristic properties of canonical systems---a general form of the ACOE together with stationary policies that solve or almost solve the ACOE. (\cite[Chap.~7.9]{DyY79} considers an abstract model with desired measurable structures; \cite{Piu89} considers the universal measurability framework.)

Gubenko and Shtatland used a contraction-based fixed point approach to prove the ACOE under, alternatively, a minorization condition and a majorization on the state transition stochastic kernel~\cite[Thms.~2, 2$'$]{GuS75} (measurability issues are assumed away in these theorems). Their majorization condition not only differs in essential ways from ours but is also too stringent to be practical (see Remark~\ref{rem-GuS-maj-cond} for details). But their approach with the minorizaton condition is a fruitful one and has close connections with ergodic theory for Markov chains; in particular, their minorization condition implies that the MDP is uniformly geometrically ergodic (cf.\ Section~\ref{sec-bd-uc}). Their contraction argument was extended by Kurano~\cite{Kur86} to a multistep-contraction argument for average-cost MDPs with u.m.\ policies. The main results of \cite{GuS75,Kur86} on the ACOE concern the case of bounded one-stage costs. 

For unbounded one-stage costs, a contraction-based fixed point approach to proving the ACOE was proposed more recently. It was originally formulated by Vega-Amaya~\cite{VAm03}, building on the result of Hern\'{a}ndez-Lerma and Lasserre on positive Harris recurrent Markov chains, as well as on the prior research of Gordienko and Hern\'{a}ndez-Lerma~\cite{GHL95}, which relates a class of MDPs with unbounded one-stage costs to $w$-geometrically ergodic Markov chains. 
In \cite{VAm03} (also the recent work \cite{VAm18}), this fixed point approach is applied to MDPs with continuity/compactness properties. The approach was taken by Ja\'{s}kiewicz~\cite{Jas09} to establish the ACOE for semi-Markovian decision processes (SMDPs) in the universal measurability framework. The conditions of \cite{Jas09,VAm03} generalize the minorization condition of \cite{GuS75}, and the ACOE results of \cite{Jas09,VAm03} are applicable to a class of MDPs that are uniformly $w$-geometrically ergodic w.r.t.~certain weight functions $w$ on the state spaces (cf.\ the discussion in Section~\ref{sec-acoe-uc}).

Meyn~\cite{Mey97} studied the convergence of policy iteration for average-cost MDPs. He obtained the ACOE and the existence of a stationary average-cost optimal policy, under a set of conditions on the Markov chains induced by stationary policies that could be generated by the policy iteration algorithm, together with a continuity condition on the differential cost functions of those policies and the one-stage cost function. A large part of his analysis is based on general state space Markov chain theory and requires no compactness/continuity conditions. However, it does not directly apply to Borel-space MDPs in the universal measurability framework due to known measurability issues in policy iteration. It is still an open problem to extend the method of~\cite{Mey97} to average-cost MDPs with u.m.\ policies.

We remark that these prior studies differ significantly from our work in both the approaches taken and the results obtained. When compactness/continuity conditions are absent, a major tool used in many of those studies is ergodic theory for Markov chains, whereas, by exploiting Egoroff's theorem, our analyses and results can be applied to non-ergodic MDPs (cf.\ the examples in Section~\ref{sec-5} and Appendix~\ref{appsec-ex-uc}).

Let us also mention two related results from our separate recent work. 
In~\cite{Yu-minpair19}, we formulated another condition of the majorization type to make use of Lusin's theorem when applying a direct method, the minimum pair approach, to study the average cost problems; this result is for countable discrete action spaces and strictly unbounded one-stage costs. 
In \cite[sect.~2.2]{Yu19} (see also Theorem~\ref{thm-ac-basic}), for the two MDP models considered in this paper, we gave a characterization of the structures of the optimal cost functions and optimal/$\epsilon$-optimal policies in the average cost problems, without any extra conditions.

The rest of this paper is organized as follows. 
In Section~\ref{sec-2}, we give background materials about Borel-space MDPs.
In Section~\ref{sec-3}, we propose new majorization type conditions and prove the ACOI for two MDP models. 
Further discussion and illustrative examples are given in Section~\ref{sec-5} and Appendix~\ref{appsec-1}.

\section{Preliminaries} \label{sec-2}

In this section, we first introduce Borel-space MDPs in the universal measurability framework. To prepare the stage for subsequent analyses, we then review several basic optimality properties for the nonnegative cost and unbounded cost models we consider, under the average and discounted cost criteria. We start with certain sets and functions that lie at the foundation of Borel-space MDPs. 

\subsection{Definitions of some Sets and Functions} \label{sec-2.1}
We consider separable metrizable spaces. A \emph{Borel space} (a.k.a.\ standard Borel space) is a separable metrizable space that is homeomorphic to a Borel subset of some Polish space \cite[Def.\ 7.7]{bs}. 
For a Borel space $X$, let $\B(X)$ denote the Borel $\sigma$-algebra and $\P(X)$ the set of probability measures on $\B(X)$ (we will call them Borel probability measures). With the topology of weak convergence, $\P(X)$ is also a Borel space~\cite[Chap.\ 7.4]{bs}. Each $p \in \P(X)$ has a unique extension on a larger $\sigma$-algebra $\B_p(X)$ generated by $\B(X)$ and all the subsets of $X$ with $p$-outer measure $0$, and this extension is called the \emph{completion of $p$} (cf.\ \cite[Chap.\ 3.3]{Dud02}). The \emph{universal $\sigma$-algebra} on $X$ is defined as $\U(X) : = \cap_{p \in \P(X)} \B_p(X)$. $\U(X)$-measurable functions on $X$ are thus measurable w.r.t.\ the completion of any Borel probability measure; these functions are called \emph{u.m.~(universally measurable)}. 

If $X$ and $Y$ are Borel spaces, a \emph{Borel or u.m.\ stochastic kernel} on $Y$ given $X$ is a function $q: X \to \P(Y)$ that is measurable from the space $\big(X, \B(X)\big)$ or $\big(X, \U(X)\big)$, respectively, to the space $\big(\P(Y), \B(\P(Y))\big)$; see \cite[Def.~7.12, Prop.~7.26, Lem.~7.28]{bs}. We use the notation $q(dy \,|\, x)$ for the stochastic kernel. When $q$ is a continuous function, we say the stochastic kernel is \emph{continuous} (a.k.a.\ \emph{weakly continuous} or \emph{weak Feller} in the literature).

An \emph{analytic set} in a Polish space is the image of a Borel subset of some Polish space under a Borel measurable function (cf.\ \cite[Prop.\ 7.41]{bs}, \cite[sect.\ 13.2]{Dud02}). A function $f: D \to [-\infty, \infty]$ is called \emph{l.s.a.~(lower semi-analytic)}, %
if $D$ is an analytic set and 
for every $a \in \R$, the level set $\{ x \in D \!\mid f(x) \leq a\}$ of $f$ is analytic \cite[Def.\ 7.21]{bs}. An equivalent definition is that the epigraph of $f$, $\{(x, a) \!\mid x \in D, f(x) \leq a, a \in \R\}$, is analytic (cf.\ \cite[p.~186]{bs}). 
For comparison, $f$ is \emph{l.s.c.~(lower semicontinuous)} if its epigraph is closed.
A Borel measurable extended-real-valued function on a Borel space is l.s.a.\ and an l.s.a.\ function is u.m., since in a Polish space every Borel set is analytic and every analytic set is u.m.\ (\cite[Cor.~7.42.1]{bs}, \cite[Thm.~13.2.6]{Dud02}).

The properties of analytic sets give rise to many properties of l.s.a.\ functions that are important for dynamic programming in Borel-space MDPs. The most critical is the Jankov-von Neumann measurable selection theorem~\cite[Prop.\ 7.49]{bs}, which asserts that for an analytic set $D$ in the product space $X \times Y$ of two Borel spaces, with $\text{proj}_X(D)$ being the projection of $D$ on $X$,
there exists an analytically measurable
\footnote{I.e., measurable w.r.t.\ the $\sigma$-algebra generated by the analytic sets.}
function $\phi: \text{proj}_X(D) \to Y$ such that the graph of $\phi$ lies in $D$.
This theorem gives rise to a measurable selection theorem for partial minimization of l.s.a.\ functions on product spaces~\cite[Prop.\ 7.50]{bs}.
For Borel-space MDPs, these properties of analytic sets and l.s.a.\ functions are closely related to the validity of value iteration, the structure of the optimal value functions, and the existence of optimal or nearly optimal policies and their structures, some of which we will discuss in Sections~\ref{sec-2.2}-\ref{sec-2.3}.
Due to space limit, however, we do not list these properties and will provide references where we use them in this paper.
\footnote{For l.s.a.\ functions, we refer the reader to the papers \cite{BFO74,MS92,ShrB78} and the monograph \cite[Chap.\ 7]{bs}; for general properties of analytic sets, see also the books \cite[Appendix~2]{DyY79} and \cite{Par67,Sriv-borel}.}

\subsection{Borel-space MDPs} \label{sec-2.2}

In the universal measurability framework, a Borel-space MDP has the following elements and model assumptions (cf.\ \cite[Chap.\ 8.1]{bs}):
\begin{itemize}[leftmargin=0.5cm,labelwidth=!]
\item The state space $\X$ and the action space $\A$ are \emph{Borel spaces}.
\item The control constraint is specified by a set-valued map $A: x \mapsto A(x)$, where for each state $x \in \X$, $A(x) \subset \A$ is a nonempty set of admissible actions at that state, and the graph of $A(\cdot)$,
$\Gamma = \{(x, a) \mid x \in \X, a \in A(x)\} \subset \X \times \A,$
is \emph{analytic}.
\item The one-stage cost function $c: \Gamma \to [-\infty, +\infty]$ is \emph{l.s.a.}
\item State transitions are governed by $q(dy \mid x, a)$, a \emph{Borel measurable} stochastic kernel on $\X$ given $\X \times \A$.
\end{itemize}

We consider infinite-horizon control problems. A policy consists of a sequence of stochastic kernels on $\A$ that specify for each stage, which admissible actions to apply, given the history up to that stage. 
In particular, a \emph{u.m.\ policy} is a sequence $\pi=(\mu_0, \mu_1, \ldots)$, where for each $k \geq 0$,
$\mu_k\big(da_k \!\mid x_0, a_0, \ldots, a_{k-1}, x_k \big)$ is a u.m.\ stochastic kernel on $\A$ given $(\X \times \A)^{k} \times \X$ and obeys the control constraint of the MDP:
\begin{equation}  \label{eq-control-constraint}
   \mu_k\big(A(x_k) \!\mid x_0, a_0, \ldots, a_{k-1}, x_k \big) = 1 \quad \forall \, (x_0, a_0, \ldots, a_{k-1}, x_k) \in (\X \times \A)^k \times \X.
\end{equation}   
(As $\Gamma$ is analytic, the sets $A(x)$ are u.m.~\cite[Lem.~7.29]{bs}; the probability of $A(x_k)$ here is measured w.r.t.\ the completion of $\mu_k(da_k \mid x_0, a_0, \ldots, a_{k-1}, x_k )$.)
A policy $\pi$ is \emph{Borel measurable} 
if each component $\mu_k$ is a Borel measurable 
stochastic kernel; $\pi$ is then also u.m.\ by definition.
(A Borel measurable policy, however, may not exist \cite{Blk-borel}.)
We define the policy space $\Pi$ of the MDP to be the set of u.m.\ policies. We shall simply refer to these policies as policies, dropping the term ``u.m.,'' if there is no confusion or no need to emphasize their measurability.

We define several subclasses of policies in the standard way.
A policy $\pi$ is \emph{nonrandomized} if $\mu_k\big(d a_k \!\mid\! x_0, a_0, \ldots, a_{k-1}, x_k \big)$ is a Dirac measure that assigns probability one to a single action in $A(x_k)$, for every $(x_0, a_0, \ldots, a_{k-1}, x_k)$ and $k \geq 0$. 
A policy $\pi$ is \emph{semi-Markov} if for every $k \geq 0$, the function $(x_0, a_0, \ldots, a_{k-1}, x_k) \mapsto  \mu_k(d a_k \!\mid\! x_0,  a_0, \ldots, a_{k-1}, x_k)$ depends only on $(x_0, x_k)$;
\emph{Markov} if for every $k \geq 0$, that function depends only on $x_k$; \emph{stationary} if $\pi$ is Markov and $\mu_k = \mu$ for all $k \geq 0$. For the stationary case, we simply write $\mu$ for $\pi = (\mu, \mu, \ldots)$. 
A nonrandomized stationary policy $\mu$ can also be viewed as a function that maps each $x \in \X$ to an action in $A(x)$, so for such $\mu$, we will use both notations $\mu(x)$, $\mu(d a\,|\, x)$ in the paper. 

Because the graph $\Gamma$ of the control constraint $A(\cdot)$ is analytic, by the Jankov-von Neumann selection theorem \cite[Prop.~7.49]{bs}, there exists at least one u.m.~nonrandomized stationary policy. Thus \emph{the policy space $\Pi$ is non-empty}.

We consider the average cost criterion and the discounted cost criterion. 
By \cite[Prop.\ 7.45]{bs}, given a policy $\pi \in \Pi$ and an initial state distribution $p_0 \in \P(\X)$, the collection of stochastic kernels $\mu_0(d a_0 \mid x_0)$, $q(dx_1 \mid x_0, a_0)$, $\mu_1(da_1 \mid x_0, a_0, x_1)$, $q(dx_2 \mid x_1, a_1), \ldots$
determines uniquely a probability measure on the universal $\sigma$-algebra on $(\X \times \A)^\infty$. The \emph{$n$-stage value function} and the \emph{average cost function} of $\pi$ are defined, respectively, by
\footnote{In general, for a u.m.\ function $f: (\X \times \A)^\infty \to [-\infty, +\infty]$, define $\E f : = \E f^+ - \E f^-$ where $f^+ = \max \{ 0, f \}$ and $f^- = - \min \{ 0, f\}$; if $\E f^+ = \E f^- = +\infty$, we adopt the convention $\infty - \infty =  - \infty + \infty =  \infty$. In the MDPs of our interest, however, we will not encounter such summations.}
$$ J_n(\pi, x) : = \E^\pi_x \Big[ \, \textstyle{\sum_{k=0}^{n-1} c(x_k, a_k)} \, \Big], \quad J(\pi,x) : = \limsup_{n \to \infty} J_n(\pi, x) /n, \quad x \in \X,$$
where $\E^\pi_x$ denotes expectation w.r.t.\ the probability measure induced by $\pi$ and the initial state $x_0 = x$.
Define 
the \emph{optimal average cost function} by
$$ g^*(x) : = \inf_{\pi \in \Pi} J(\pi,x) = \inf_{\pi \in \Pi} \limsup_{n \to \infty} J_n(\pi, x)/n, \qquad x \in \X.$$
For $0 < \alpha < 1$, define the \emph{$\alpha$-discounted value function} of a policy $\pi$ by
$$ v^\pi_\alpha(x) : = \limsup_{n \to \infty}  \E^\pi_x \Big[ \, \textstyle{ \sum_{k=0}^{n-1} \alpha^k c(x_k, a_k) } \, \Big], \qquad x \in \X,$$
and the \emph{optimal $\alpha$-discounted value function} by
$$ v_\alpha(x) : = \inf_{\pi \in \Pi} v^\pi_\alpha(x) , \qquad x \in \X.$$
The functions $J_n(\pi, \cdot)$, $J(\pi, \cdot)$ and $v^\pi_\alpha(\cdot)$ are u.m.\ by \cite[Prop.\ 7.46, Lem.~7.30(2)]{bs}. 
As will be discussed in the next subsection, for the two MDP models we consider, the optimal cost functions $g^*$ and $v_\alpha$ are l.s.a.

Let $\M(\X)$ (resp.\ $\An(\X)$) denote the space of extended-real-valued u.m.\ (resp.\ l.s.a.)~functions on $\X$.
For $0 < \alpha \leq 1$, define \emph{dynamic programming operators} $T_\alpha$ that map $v \in \M(\X)$ to a function on $\X$ according to
\footnote{Here and throughout the paper, the integration of a u.m.\ function w.r.t.\ $p \in \P(\X)$ is defined w.r.t.\ the completion of $p$ (cf.\ Section~\ref{sec-2.1}).}
$$(T_\alpha v)(x) : = \inf_{a \in A(x)} \left\{ c(x, a) + \alpha \int_\X v(y) \, q(dy \mid x, a) \right\}, \qquad x \in \X.$$
For $\alpha = 1$, we simply write $T$ for $T_\alpha$. 
By the properties of analytic sets and l.s.a.\ functions (cf.~\cite[Chap.\ 7]{bs}), $T_\alpha$ and $T$ map $\An(\X)$ into $\An(\X)$.

The following subclasses of u.m.\ and l.s.a.\ functions will be needed shortly. For a u.m.~function $w : \X \to (0, +\infty)$, which we shall refer to as a \emph{weight function}, let 
$$\M_w(\X) : = \big\{ f \mid  \| f \|_w < \infty, f \in \M(\X) \big\}, \quad \text{where} \  \ \| f\|_w : = \sup_{x \in \X} \big| f(x) \big| /w(x).$$
Note that $\big(\M_w(\X), \|\cdot\|_w\big)$ is a Banach space, and $\An(\X) \cap \M_w(\X)$ is a closed subset of this space.
\footnote{This is because convergence in the $\|\cdot\|_w$ norm implies pointwise convergence, and pointwise limits of a sequence of l.s.a.\ functions are l.s.a.~\cite[Lem.~7.30(2)]{bs}.}

\subsection{Some Basic Optimality Properties of Two Models} \label{sec-2.3}

We consider two classes of MDPs. The first is the nonnegative cost model where the one-stage cost function $c \geq 0$. We shall refer to this model as (PC) in what follows. For the average cost or discounted problem, it is equivalent to the case where $c$ is bounded from below. 

The second model, designated as (UC), involves unbounded costs: the function $c$ can be unbounded from below or above, but it needs to satisfy a growth condition and moreover, there is a Lyapunov-type condition on the dynamics of the MDP. The precise definition is as follows.
\begin{definition}[the model (UC)] \label{def-uc}
There exist a u.m.\ weight function $w(\cdot) \geq 1$ and constants $b, \hat c \geq 0$ and $\lambda \in [0,1)$ such that for all $x \in \X$,
\begin{enumerate}[leftmargin=0.65cm,labelwidth=!]
\item[\rm (a)] $\sup_{a \in A(x)} |c(x, a)| \leq \hat c \, w(x)$;
\item[\rm (b)] $\sup_{a \in A(x)} \int_\X w(y) \, q(dy \mid x, a) \leq \lambda w(x) + b$.
\end{enumerate}
\end{definition}

For (UC), its definition ensures that the average cost function of any policy $\pi$ satisfies $\|J(\pi, \cdot) \|_w \leq \ell$ for the constant $\ell =  \hat c \,b /(1-\lambda)$. Hence the optimal average cost function also satisfies $\|g^*\|_w \leq \ell$ and in particular, $g^*$ is finite everywhere. For (PC), $g^* \geq 0$ and it is possible that at some state $x$, $g^*(x) = +\infty$ (this possibility will be eliminated in Section~\ref{sec-3} under further assumptions on the MDP model).

For both (PC) and (UC), the average-cost MDPs have the general optimality properties given in the next theorem, which is proved by the author \cite[sects.~2.2,~A.2]{Yu19}. In what follows, by an $\epsilon$-optimal or optimal policy (with no mention of an initial state), we mean a policy that is $\epsilon$-optimal or optimal \emph{for all initial states}. 

{\samepage
\begin{theorem}[average-cost optimality results; {\cite[Thm.~2.1]{Yu19}}] \label{thm-ac-basic}
{\rm (PC)(UC)}
\begin{enumerate}[leftmargin=0.7cm,labelwidth=!]
\item[\rm (i)] The optimal average cost function $g^*$ is l.s.a. 
\item[\rm (ii)] For each $\epsilon > 0$, there exists a (u.m.)~randomized semi-Markov $\epsilon$-optimal policy. If there exists an optimal policy for each state $x \in \X$, then there exists a (u.m.)~randomized semi-Markov optimal policy.
\end{enumerate}
\end{theorem}
}

In Section~\ref{sec-3}, we will use the vanishing discount factor approach to prove the ACOI for (PC) and (UC) under additional conditions.
That analysis starts with the optimality equations for the $\alpha$-discounted cost criteria ($\alpha$-DCOEs) given below:

\begin{theorem}[the $\alpha$-DCOE] 
\label{thm-dcoe} {\rm (PC)(UC)} 
For $\alpha \in (0,1)$, the optimal value function $v_\alpha$ is l.s.a.\ and satisfies the $\alpha$-DCOE $v_\alpha = T_\alpha v_\alpha$, i.e.,
$$  v_\alpha (x) =  \inf_{a \in A(x)} \left\{ c(x, a) + \alpha \int_\X v_\alpha(y) \, q(dy \mid x, a) \right\}, \qquad x \in \X.$$
For (PC), $v_\alpha$ is the smallest nonnegative solution of the $\alpha$-DCOE in $\An(\X)$; for (UC), $v_\alpha$ is the unique solution of the $\alpha$-DCOE in the space $\An(\X) \cap \M_w(\X)$. Furthermore, in both cases, for each $\epsilon > 0$, there exists a nonrandomized stationary $\epsilon$-optimal policy. 
\end{theorem}

\begin{rem}[about the proof of Theorem~\ref{thm-dcoe}]  \label{remark-dcoe}
Under certain compactness and continuity conditions, proofs of the $\alpha$-DCOE for (PC) and (UC) can be found in e.g., the papers~\cite{FKZ12,Sch75} and the books \cite[Chap.\ 5]{HL96}, \cite[Chap.\ 8]{HL99}. In our case, Theorem~\ref{thm-dcoe} is proved by using the results of \cite[Part II]{bs} for general Borel-space MDPs. 
Specifically, for (PC), this theorem is implied by the optimality results for the nonnegative model~\cite[Props.\ 9.8, 9.10, 9.19]{bs}.
For (UC), it can be shown
\footnote{See \cite[Lem.~A.2, Appendix~A.4]{Yu19} for the proof. The construction of the weight function $\tilde w \geq w$ and the proof of this contraction property are similar to the analysis given in \cite[pp.\ 45-46]{HL99}.}
that for some u.m.\ weight function $\tilde w \geq w$, the operator $T_\alpha$ is a contraction on the closed subset $\An(\X) \cap \M_{\tilde w}(\X)$ of the Banach space $(\M_{\tilde w}(\X), \| \cdot\|_{\tilde w})$, with contraction modulus $\beta \in (\alpha, 1)$.
We use this contraction property of $T_\alpha$ 
together with the correspondence between the MDP and the so-called deterministic control model (DM) defined in~\cite[Chap.\ 9]{bs} to prove Theorem~\ref{thm-dcoe} for (UC). The proof can be found in \cite[Appendix~A.3]{Yu19}.
\footnote{The proof is similar to, but does not follow exactly the one given in \cite[Chap.\ 9]{bs} for bounded one-stage costs; see \cite[Remark~A.1, Appendix~A.3]{Yu19} for further explanations, including when one can reduce the case to that of bounded costs by using Veinott's similarity transformation \cite{vdW84,Vei69}.}
\myqed
\end{rem}

As another preparation for Section~\ref{sec-3}, the next lemma states an implication of the ACOI on the existence and structure of average-cost optimal or nearly optimal policies. For comparison, note that in the general case where $g^*$ need not be constant, one can only assert the existence of a randomized semi-Markov $\epsilon$-optimal policy (cf.~Theorem~\ref{thm-ac-basic}). The proof of this lemma uses mostly standard arguments and can be found in~\cite[Appendix~A.4]{Yu19}.

\begin{lemma}[a consequence of ACOI] \label{lem-optpol}
Consider the models (PC) and (UC) with the average cost criterion. 
Suppose that the optimal average cost function $g^*$ is constant and finite. 
Suppose also that for some real-valued $h \in \An(\X)$, with $h \geq 0$ for (PC) and $\|h\|_w < \infty$ for (UC), the ACOI holds: $g^* + h \geq T h$, i.e.,
\begin{equation} \label{eq-acoi-gen}
   g^* + h(x) \geq \inf_{a \in A(x)} \left\{ c(x, a) + \int_\X h(y) \, q(dy \mid x, a) \right\}, \qquad x \in \X.
\end{equation}   
Then there exist a nonrandomized Markov optimal policy and, for each $\epsilon > 0$, a nonrandomized stationary $\epsilon$-optimal policy. If, in addition, the infimum in the right-hand side of the ACOI is attained for every $x \in \X$, then there exists a nonrandomized stationary optimal policy. 
\end{lemma}

\section{Main Results: The ACOI} \label{sec-3}

In this section, we introduce our majorization condition for the (PC) and (UC) models and study the ACOI via the vanishing discount factor approach. The arguments for (PC) and (UC) are similar but differ in details, so we will discuss the two models separately. We consider (UC) first.

\subsection{The (UC) Case} \label{sec-3.1}

Let $\bar x$ be some fixed state and consider the relative value functions of the $\alpha$-discounted problems:
\begin{equation} \label{eq-uc-ha}
h_\alpha(x) : = v_\alpha(x) - v_\alpha(\bar x), \qquad x \in \X.
\end{equation}
Take a sequence $\alpha_n \uparrow 1$ such that for some $\rho^* \in \R$,
\begin{equation} \label{eq-rho}
  (1 - \alpha_n) \, v_{\alpha_n}(\bar x) \to \rho^* \quad \text{as} \  n \to \infty.
\end{equation}  
(This is possible because the model conditions of (UC) imply that for each $x \in \X$, $(1 - \alpha) \, v_\alpha (x)$ is bounded in $\alpha$ over $(0,1)$; cf.\ \cite[sect.~10.4.A]{HL99}.)
Consider the corresponding sequence of functions 
$h_n : = h_{\alpha_n}$. 
Define 
\begin{align}
 \uh  : = \liminf_{n \to \infty} h_n, \qquad  \uh_n : = \inf_{m \geq n} h_m, \ \ \ n \geq 0,  \label{eq-uh} \\
 \bh  : = \limsup_{n \to \infty} h_n, \qquad  \bh_n : = \sup_{m \geq n} h_m, \ \ \ n \geq 0. \label{eq-bh}
\end{align}
Note that as $n \to \infty$,
$\uh_n \uparrow \uh$ and $\bh_n \downarrow \bh.$

\begin{assumption} \label{cond-uc-1}
For the model (UC), the set of functions $\{h_\alpha \mid \alpha \in (0,1)\}$ as defined above is bounded in $\M_w(\X)$, i.e., $\sup_{\alpha \in (0,1)} \|h_\alpha\|_w < \infty$.
\end{assumption}

This assumption is extracted from the analysis of the ACOI given in the book~\cite[Chap.\ 10.4]{HL99}, where the MDP is assumed to be uniformly $w$-geometrically ergodic, which ensures $\sup_{\alpha \in (0,1)} \|h_\alpha\|_w < \infty$. In Section~\ref{sec-bd-uc} we shall discuss this $w$-geometric ergodicity condition as well as other types of sufficient conditions for Assumption~\ref{cond-uc-1}.  

The next lemma follows directly from Theorem~\ref{thm-dcoe} and \cite[Lem.~7.30(2)]{bs}:

\begin{lemma} \label{lem-h-lsa-uc}
{\rm (UC)} Under Assumption~\ref{cond-uc-1}, 
all the functions $\uh$, $\bh$, $\uh_n$, $\bh_n$, $n \geq 0$, are l.s.a.\ and lie in a bounded subset of $\M_w(\X)$.
\end{lemma}

\subsubsection{Majorization Condition} \label{sec-uc-cond}

We now introduce our majorization condition for the (UC) model. The first two parts of this condition will also be used for the (PC) model later. Let $\ind(\cdot)$ denote the indicator function.

\begin{assumption} \label{cond-uc-2}
In (UC), for each $x \in \X$ and $\epsilon > 0$, the following hold:
\begin{itemize}[leftmargin=0.6cm,labelwidth=!]
\item[\rm (i)] There exist a subset $K_\epsilon(x) \subset A(x)$ and $0 < \bar \alpha  < 1$ such that for all $\alpha \in [\bar \alpha, 1)$,
\begin{equation} \label{cond-uc-2a}
 \inf_{a \in K_\epsilon(x)} \left\{ c(x, a) + \alpha \int_\X v_\alpha(y) \, q(dy \mid x, a) \right\} \leq v_\alpha(x) + \epsilon.  
\end{equation} 
\item[\rm (ii)] There exists a finite measure $\nu$ on $\B(\X)$ such that 
\begin{equation}
 \sup_{a \in K_\epsilon(x)} q( B \mid x, a) \leq \nu(B) \qquad \forall \, B \in \B(\X). \label{cond-uc-2b} 
\end{equation}
\item[\rm (iii)] The weight function $w(\cdot)$ is uniformly integrable w.r.t.\ $\{ q(dy \,|\, x, a) \mid a \in K_\epsilon(x) \}$ in the sense that
\begin{equation}
   \lim_{\ell \to \infty} \sup_{a \in K_\epsilon(x) }  \int_\X w(y) \, \ind\big[w(y) \geq \ell \, \big] \, q(dy \mid x, a) = 0. \label{cond-uc-2c}
\end{equation}
\end{itemize}
\end{assumption}

Note that in the above, $\bar \alpha$ and $\nu$ can depend on $x$ and $\epsilon$. Note also that the one-stage cost function $c(\cdot)$ is not required to have any special properties.  

Assumption~\ref{cond-uc-2}(i) is motivated by the theory on epi-convergence of functions, in particular, by the relation between the infima of epi-limits and pointwise limits of a sequence of functions and the implication of this relation on the validity of the ACOI in MDPs. In some circumstances, the existence of a compact subset $K_\epsilon(x)$ with the property (\ref{cond-uc-2a}) is close to being necessary for the ACOI. We shall elaborate on this in Section~\ref{sec-epilim} (cf.\ Proposition~\ref{prp-inf-epilim}, Cor.~\ref{cor-epilim-acoi}) after we prove the ACOI, since the roles of the above conditions can be better seen then.

The choice of the subset $K_\epsilon(x)$ of actions in Assumption~\ref{cond-uc-2}(i) is important for the subsequent parts (ii)-(iii) of this assumption. Although (i) is trivially satisfied if we let $K_\epsilon(x) = A(x)$, this makes it harder or sometimes impossible to satisfy (ii)-(iii). For instance, if $A(x) = \R$ and $q(dy \,|\, x, a)$ is the normal distribution $\mathcal{N}(a, 1)$ on $\R$, no finite measure can majorize these distributions for all $a \in \R$. On the other hand, identifying a proper subset of actions that satisfies (i) is in general a difficult problem, because it requires the knowledge of the family of optimal value functions $\{v_\alpha\}$. An additional difficulty in the (UC) case is that, unlike in the (PC) case, the cost function $c(x, \cdot)$ here must be bounded for each state $x$, preventing us from using ``coercivity'' of $c(x,\cdot)$ to eliminate actions. (For this reason, in the illustrative example given in Appendix~\ref{appsec-ex-uc} for the (UC) model, we will set $K_\epsilon(x) = A(x)$.)

Assumption~\ref{cond-uc-2}(ii) is the key condition that allows us to apply Egoroff's theorem in proving the ACOI. The class of MDPs for which Assumption~\ref{cond-uc-2}(ii) can hold naturally are those where for all $a \in K_\epsilon(x)$, $q(dy \,|\, x, a)$ has a density $f_{x,a}$ w.r.t.\ a common ($\sigma$-finite) reference measure $\varphi$. The pointwise supremum of the density functions, $f_x: = \sup_{a \in K_\epsilon(x)} f_{x,a}$ (or a measurable function that upper-bounds it), when it is integrable w.r.t.\ $\varphi$, 
defines a finite measure $\nu$, with $d\nu = f_x d \varphi$, that has the desired majorization property (\ref{cond-uc-2b}). 

Assumption~\ref{cond-uc-2}(ii) need not be satisfied by continuous state transition stochastic kernels. For example, if $A(x)=[0,1]$ and $q(dy \,|\, x, a) = \delta_a$ (the Dirac measure at $a$), there is no finite measure with the desired majorization property (there is also no nontrivial measure $\nu$ that can satisfy the minorization condition $q(dy \,|\, x, a) \geq \nu(dy)$ for all $a$ here). This is not a ``defect'' in the majorization condition but a reflection of the difference in nature of topological and measure-theoretic properties; a deeper discussion of Assumption~\ref{cond-uc-2}(ii) will be given in Section~\ref{sec-maj-wsc} to show what it entails. Due to this difference, however, there are some inevitable limitations in the analysis technique we present in this paper. 
We shall discuss this subject further in the example in Section~\ref{sec-ex-pc} (cf.\ Remark \ref{rmk-maj-density}).

Regarding Assumption~\ref{cond-uc-2}(iii), roughly speaking, in the proof of the ACOI, we will use it to handle what is leftover after the application of Egoroff's theorem. Verifying this condition can be straightforward when the weight function $w(\cdot)$ has a simple analytical expression (e.g., when $x \in \R$ and $w(x) = e^x$ or $x^2$) and the properties of the integrals involved are known. 
For example, if $w(\cdot)$ is bounded from above on the union of the supports of the probability measures $q(dy \mid x, a), a \in K_\epsilon(x)$, such as in the case where the union is contained in a compact set and $w(\cdot)$ is continuous, then Assumption~\ref{cond-uc-2}(iii) is clearly satisfied. If $c(\cdot)$ is bounded, $w(\cdot)$ can be chosen to be constant and Assumption~\ref{cond-uc-2}(iii) then holds trivially. 
Assumption~\ref{cond-uc-2}(iii) is also implied by the slightly stronger condition $\int w \, d \nu < \infty$, where $\nu$ is the majorizing finite measure in Assumption~\ref{cond-uc-2}(ii). 

Let us end this discussion by clarifying the relation between our majorization condition and the one from Gubenko and Shtatland's early work. Later we will discuss further the conditions involved in our proof of the ACOI and give illustrative examples in Section~\ref{sec-5} and Appendix~\ref{appsec-ex-uc}.

\begin{rem}[about the majorization condition in {\cite{GuS75}}] \label{rem-GuS-maj-cond}
In their contraction-based fixed point approach, Gubenko and Shtatland's majorization condition ~\cite[sect.~3, Condition~(II)]{GuS75} is like a symmetric counterpart of their minorization condition, and it requires that there exists a finite measure $\nu$ on $\B(\X)$ such that
\begin{equation} \label{eq-GuS-cond}
  q(B \mid x, a) \leq \nu(B) \quad \forall \, B \in \B(\X), \ (x, a) \in \Gamma,  \qquad \text{and} \qquad \nu(\X) < 2. 
\end{equation}  
Note that here the same measure $\nu$ needs to majorize $q(dy \,|\, x, a)$ for all states and admissible actions, whereas in our Assumption~\ref{cond-uc-2}, $\nu$ can be different for each state. The requirement $\nu(\X) < 2$ (needed for converting $T$ into a contraction) is too stringent and renders their condition (\ref{eq-GuS-cond}) impractical.\myqed
\end{rem}

\subsubsection{Optimality Results} \label{sec-uc-acoi-prf}

We now prove the ACOI for (UC) under the assumptions introduced earlier. 
Recall that the relative value functions $\{h_n\}$, the functions $\uh$, $\{\uh_n\}$, $\bh$, $\{\bh_n\}$, and also the scalar $\rho^* = \lim_{n \to \infty} (1 - \alpha_n) \, v_{\alpha_n} (\bar x)$ are defined in~(\ref{eq-uc-ha})-(\ref{eq-bh}).

\begin{theorem}[the ACOI for (UC)] \label{thm-uc-acoi}
Under Assumptions~\ref{cond-uc-1}, \ref{cond-uc-2} for the (UC) model, the optimal average cost function $g^*(\cdot) = \rho^*$ and with $\uh \in \An(\X) \cap \M_w(\X)$ as given in (\ref{eq-uh}), the pair $(\rho^*, \uh)$ satisfies the ACOI:
\begin{equation} \label{eq-uc-acoi}
  \rho^* + \uh(x) \geq \inf_{a \in A(x)} \left\{ c(x, a) + \int_\X \uh(y) \, q(dy \mid x, a) \right\}, \qquad x \in \X.
\end{equation}  
Hence there exist a nonrandomized Markov optimal policy and for each $\epsilon > 0$, a nonrandomized stationary $\epsilon$-optimal policy. 
\end{theorem}

Note that this theorem implies that $\rho^*$ does not depend on our choice of the sequence $\{\alpha_n\}$ or the state $\bar x$, and
$\lim_{\alpha \to 1} (1 - \alpha) \, v_\alpha(x) = g^*$ (the constant optimal average cost) for all $x \in \X$. 

This theorem can be compared with the prior ACOI result for (UC) with compactness/continuity conditions given in the book~\cite[Thm.~10.3.1]{HL99}.

In the case of bounded one-stage costs, another existing result due to Ross \cite[Thm.~2]{Ros68} says that when $\A$ is finite, if $\{h_\alpha\}$ is uniformly bounded and equicontinuous, then the ACOE holds for a continuous function $h$ and constant $g^*$. For comparison, treating the case of bounded one-stage costs as a special case of (UC) with $w(\cdot) \equiv 1$, we have the following corollary from Theorem~\ref{thm-uc-acoi}:

\begin{cor} \label{cor-bc-acoi}
In an MDP with bounded $c(\cdot)$, suppose that for each $x \in \X$, there is a finite measure $\nu$ on $\B(\X)$ such that $\sup_{a \in A(x)} q(B \,|\, x, a) \leq \nu(B)$ for all $B \in \B(\X)$. Then, if $\{h_\alpha\}$ is uniformly bounded, the ACOI (\ref{eq-uc-acoi}) holds.
\end{cor}

\begin{rem} \label{rem-uc-acoe}
It will be obvious from the proof of Theorem~\ref{thm-uc-acoi} that if we have $\uh = \lim_{n \to \infty} h_n$ in addition, then the ACOI (\ref{eq-uc-acoi}) holds with equality.
The extra condition on the convergence of $\{h_n\}$ can be stringent and hard to check in practice, however.
A discussion about existing results on the ACOE in the (UC) case and whether Theorem~\ref{thm-uc-acoi} can be strengthened to ACOE will be given in Section~\ref{sec-bd-uc}.
 \myqed
\end{rem}

We now proceed to prove Theorem~\ref{thm-uc-acoi}. First, we prove an important lemma, in which we combine the majorization condition with Egoroff's theorem. (An alternative proof of this lemma will be given in Section~\ref{sec-maj-wsc}.)

\begin{lemma} \label{lem-uc-2acoi}
{\rm (UC)} Let Assumptions~\ref{cond-uc-1},~\ref{cond-uc-2}(ii)-(iii) hold, and let $K_\epsilon(x) \subset A(x)$ be the set in the latter condition for a given $x \in \X$ and $\epsilon > 0$. Then
\begin{align*}
& \lim_{n \to \infty} \inf_{a \in K_\epsilon(x)} \left\{ c(x, a) + \alpha_n \int_\X \uh_n(y) \, q(dy \mid x, a) \right\}  
 = \inf_{a \in K_\epsilon(x)} \left\{ c(x, a) +  \int_\X \uh(y) \, q(dy \mid x, a) \right\}.
\end{align*}
\end{lemma} 

\begin{proof} 
For the state $x$, $\epsilon > 0$, and the set $K_\epsilon(x)$ given in the lemma, let $\nu$ be the corresponding finite measure on $\B(\X)$ in Assumption~\ref{cond-uc-2}(ii). Recall that $\uh_n \uparrow \uh$ and these functions are real-valued and u.m.\ (Lemma~\ref{lem-h-lsa-uc}). Therefore, 
by Egoroff's Theorem~\cite[Thm.~7.5.1]{Dud02},
\footnote{Egoroff's Theorem~\cite[Thm.~7.5.1]{Dud02}: Let $(X,\mathcal{S},\nu)$ be finite measure space. Let $f_n$ and $f$ be measurable functions from $X$ into a separable metric space. Suppose $f_n(x) \to f(x)$ for $\nu$-almost all $x$. Then for any $\delta > 0$, there is a set $D$ with $\nu(X \setminus D) < \delta$ such that $f_n \to f$ uniformly on $D$.\label{footnote-Egroff}}
for any $\delta > 0$, there exists a u.m.\ set $D_\delta \subset \X$ with $\nu(\X \setminus D_\delta) < \delta$ such that on the set $D_\delta$, $\uh_n$ converges to $\uh$ uniformly as $n \to \infty$.
Consequently, for any $\eta > 0$, it holds for all $n$ sufficiently large that
\begin{equation} \label{eq-prf-thm-uc1a}
   \int_{D_\delta} \big(\uh(y) - \uh_n(y) \big)\, q(dy \mid x, a) \leq  \eta \qquad \forall \, a \in  A(x).
\end{equation}
We now bound the integral of $\uh - \uh_n$ on the complement set $\X \setminus D_\delta$. 
By Lemma~\ref{lem-h-lsa-uc}, for all $n \geq 0$, $\| \uh - \uh_n \|_w \leq \ell$ for some constant $\ell$. So for all $a \in A(x)$,
$$ \int_{\X \setminus D_\delta} \big(\uh(y) - \uh_n(y) \big)\, q(dy \mid x, a) \, \leq  \, \ell  \int_{\X \setminus D_\delta}  w(y) \, q(dy \mid x, a).$$
By the choice of $D_\delta$ and the majorization property of $\nu$ in Assumption~\ref{cond-uc-2}(ii),
\begin{equation} \label{eq-prf-thm-uc1b}
 \sup_{a \in K_\epsilon(x)} q\big(\X \setminus D_\delta \mid x, a\big)  \leq \nu\big(\X \setminus D_\delta\big) < \delta.
\end{equation} 
By an alternative characterization of uniform integrability~\cite[Thm.~10.3.5]{Dud02}, (\ref{eq-prf-thm-uc1b}) together with the uniform integrability condition in Assumption~\ref{cond-uc-2}(iii) implies that for any given $\eta > 0$, it holds for all $\delta$ sufficiently small that 
$\int_{\X \setminus D_\delta}  w(y) \, q(dy \mid x, a) \leq \eta$ for all $a \in K_\epsilon(x)$.
Therefore, given $\eta > 0$, by choosing a small enough $\delta$, we can make
\begin{equation} \label{eq-prf-thm-uc1c}
 \sup_{a \in K_\epsilon(x)} \int_{\X \setminus D_\delta} \big(\uh(y) - \uh_n(y) \big)\, q(dy \mid x, a) \leq \eta.
\end{equation} 
Combining this with (\ref{eq-prf-thm-uc1a}), we obtain that for all $n$ sufficiently large,
$$ \sup_{a \in K_\epsilon(x)}  \int_\X  \big(\uh(y) - \uh_n(y) \big)\, q(dy \mid x, a) \leq 2 \eta$$
and hence
\begin{align*}
  & \inf_{a \in K_\epsilon(x)} \left\{ c(x, a) + \alpha_n \int_\X \uh_n(y) \, q(dy \mid x, a) \right\} 
  \geq \inf_{a \in K_\epsilon(x)} \left\{ c(x, a) + \alpha_n \int_\X \uh(y) \, q(dy \mid x, a) \right\} - 2 \eta.
\end{align*}  
Since $\eta$ is arbitrary and $\uh_n \leq \uh$, the lemma follows by letting $n \to \infty$ on both sides of the preceding inequality and using also the fact that 
$$ (1 - \alpha_n) \sup_{a \in K_\epsilon(x)}  \int_\X |\uh(y)| \, q(dy \mid x, a) \to 0$$
since $\sup_{a \in K_\epsilon(x)}  \int_\X |\uh(y) | \, q(dy \mid x, a) < \infty$ by Assumption~\ref{cond-uc-1} and the model condition of (UC). 
\end{proof}

\begin{proof}[Proof of Theorem~\ref{thm-uc-acoi}]
For each $x \in \X$ and $\epsilon > 0$, by the $\alpha$-DCOE (Theorem~\ref{thm-dcoe}) and Assumption~\ref{cond-uc-2}(i), for all $n$ sufficiently large
\begin{align}
   ( 1 - \alpha_n) \, v_{\alpha_n}(\bar x) + h_n(x) \, & = \, \inf_{a \in A(x)} \left\{ c(x, a) + \alpha_n \int_\X h_n(y) \, q(dy \mid x, a) \right\} \label{eq-prf-thm-uc-oe} \\
   & \geq \,  \inf_{a \in K_\epsilon(x)} \left\{ c(x, a) + \alpha_n \int_\X h_n(y) \, q(dy \mid x, a) \right\} - \epsilon \notag \\
   & \geq \,  \inf_{a \in K_\epsilon(x)} \left\{ c(x, a) + \alpha_n \int_\X \uh_n(y) \, q(dy \mid x, a) \right\} - \epsilon, \label{eq-prf-thm-uc0a}
\end{align}   
where the last inequality used the fact $\uh_n \leq h_n$ and that $\uh_n$ is u.m.\ (Lemma~\ref{lem-h-lsa-uc}).
Letting $n \to \infty$ in both sides of (\ref{eq-prf-thm-uc0a}),
we have 
\begin{align*}
 \rho^* + \uh(x) + \epsilon & \, \geq \,  \liminf_{n \to \infty} \inf_{a \in K_\epsilon(x)} \left\{ c(x, a) + \alpha_n \int_\X \uh_n(y) \, q(dy \mid x, a) \right\} \\
   & \, = \,  \inf_{a \in K_\epsilon(x)} \left\{ c(x, a) +  \int_\X \uh(y) \, q(dy \mid x, a) \right\}  \\
   & \, \geq \, \inf_{a \in A(x)} \left\{ c(x, a) +  \int_\X \uh(y) \, q(dy \mid x, a) \right\},
\end{align*}
where the equality follows from Lemma~\ref{lem-uc-2acoi}. Since this holds for every $x \in \X$ and $\epsilon$ is arbitrary, the desired inequality (\ref{eq-uc-acoi}) is proved.

To show $g^*(\cdot) = \rho^*$, as in the analysis in~\cite{JaN06}, it suffices to show that for the pair $(\rho^*, \bh)$, where $\bh = \limsup_{n \to \infty} h_n$ as we recall, the opposite inequality holds:
\begin{equation} \label{eq-prf-thm-uc-2a} 
    \rho^* + \bh(x) \leq \inf_{a \in A(x)} \left\{ c(x, a) + \int_\X \bh(y) \, q(dy \mid x, a) \right\}, \qquad x \in \X.
\end{equation}    
This will yield $g^*(\cdot) = \rho^*$ because the inequality (\ref{eq-prf-thm-uc-2a}) implies that for every state $x$, every policy $\pi$ has average cost $J(\pi,x) \geq \rho^*$, whereas the inequality (\ref{eq-uc-acoi}) just proved implies the existence of a policy $\pi$ with $J(\pi,x) \leq \rho^*$ (the proofs of these two facts are standard).
From the $\alpha$-DCOE (\ref{eq-prf-thm-uc-oe}), we have that for all $a \in A(x)$, 
\begin{equation}
( 1 - \alpha_n) \, v_{\alpha_n}(\bar x) + h_n(x) \, \leq \, c(x, a) + \alpha_n \int_\X h_n(y) \, q(dy \mid x, a),
\end{equation}
so by taking limit supremum of both sides as $n \to \infty$ and using also the fact $\bh_n \geq h_n$ by definition, we have
$$ \rho^* + \bh(x) \leq c(x, a) + \limsup_{n \to \infty} \alpha_n \int_\X \bh_n(y) \, q(dy \mid x, a) = c(x, a) +  \int_\X \bh(y) \, q(dy \mid x, a),$$
where the last equality follows from the dominated convergence theorem, in view of Lemma~\ref{lem-h-lsa-uc} and the model condition of (UC). 
This proves (\ref{eq-prf-thm-uc-2a}) and hence $g^*(\cdot) = \rho^*$ as discussed earlier. Finally, that $\uh \in \An(\X) \cap \M_w(\X)$ follows from Lemma~\ref{lem-h-lsa-uc}, and the existence of a Markov optimal policy and stationary $\epsilon$-optimal policy follows from the ACOI proved above and Lemma~\ref{lem-optpol}.
\end{proof}

The simple ACOE result mentioned in Remark~\ref{rem-uc-acoe} follows from the ACOI just proved and the inequality (\ref{eq-prf-thm-uc-2a}) in the proof above since $\uh = \bh$ in that case.

\subsection{The (PC) Case} \label{sec-pc-acoi}

We now consider the nonnegative cost model (PC). Let 
$$\textstyle{m_\alpha : = \inf_{x \in \X} v_\alpha(x)}, \qquad \alpha \in (0,1).$$
We shall work under the following assumption taken from the ACOI literature:

\begin{assumption} \label{cond-pc-1}
For the model (PC): 
\begin{enumerate}[leftmargin=0.6cm,labelwidth=!]
\item[\rm (G)] For some policy $\pi$ and state $x$, the average cost $J(\pi,x) < \infty$.
\item[\rm (\underline{B})] For every $x \in \X$, $\liminf_{\alpha \uparrow 1} (v_\alpha(x) - m_\alpha) < \infty$.
\end{enumerate}
\end{assumption}

\begin{rem} 
Precursors to these conditions were introduced in the early works of Sennott \cite{Sen89} and Sch{\"a}l~\cite{Sch93} and then evolved as the research progressed. 
The condition (G) is the same as that in~\cite{Sch93} and by \cite[Lem.~1.2(b)]{Sch93}, it implies that 
$$ \textstyle{ \limsup_{\alpha \to 1} \, (1 - \alpha) \, m_\alpha \leq \inf_{x \in \X} g^*(x) < \infty.}$$ 
The condition (\underline{B}) was introduced by Feinberg, Kasyanov, and Zadoianchuk~\cite{FKZ12} to weaken the condition (B) in \cite{Sch93}, which is 
\begin{equation} \label{pc-cond-B}
 \textstyle{\sup_{\alpha \in (0,1)}  \big( v_\alpha(x) - m_\alpha \big) = \sup_{\alpha \in (0,1)} h_\alpha(x) < \infty} \qquad \forall \, x \in \X.
\end{equation} 
Under (G), (\ref{pc-cond-B}) is equivalent to $\limsup_{\alpha \uparrow 1} h_\alpha(x) < \infty$ for $x \in \X$~\cite[Lem.~5]{FKZ12}.
Sennott first showed in her work on countable-space MDPs~\cite{Sen89} that in proving the ACOI, it suffices to require the family of relative value functions to be bounded from above pointwise instead of uniformly.
\myqed
\end{rem}

Let 
\begin{equation} \label{eq-pc-uh}
    h_\alpha : = v_\alpha - m_\alpha, \qquad  \uh : = \liminf_{\alpha \to 1} h_\alpha.
\end{equation}    
For each $\alpha \in (0, 1)$, define a function $\uh_\alpha$ as 
\begin{equation}
    \uh_\alpha(x) : = \inf_{\beta \in [\alpha, 1)} h_\beta(x), \qquad x \in \X.
\end{equation}    
Note that
\begin{equation}
\uh_\alpha \leq \uh_\beta \quad \forall \, \alpha \leq \beta < 1 , \qquad \text{and} \qquad \uh_\alpha \uparrow \uh \quad \text{as} \ \alpha \uparrow 1.
\end{equation}
Before proceeding, let us prove that these functions are l.s.a., so that we can take their integrals in the subsequent analysis.

\begin{lemma} \label{lem-h-lsa-pc}
{\rm (PC)} Under Assumption~\ref{cond-pc-1}, the functions $\uh$ and $\uh_\alpha$, $\alpha \in (0, 1)$, are real-valued and l.s.a.
\end{lemma}

\begin{proof}
Clearly, the functions $\uh$ and $\uh_\alpha$ are real-valued under Assumption~\ref{cond-pc-1}(\underline{B}). In order to prove that they are l.s.a., let us treat the relative value functions $h_\alpha$, $\alpha \in (0,1)$, as a function of $(\alpha, x)$. We have proved in \cite[Lem.~A.1, Appendix~A.4]{Yu19} that $v_\alpha(x)$ is an l.s.a.\ function of $(\alpha, x)$ on $(0, 1) \times \X$ (this proof uses the deterministic control model corresponding to the MDP mentioned in Remark~\ref{remark-dcoe} in Section~\ref{sec-2.3}).
Next, consider $m_\alpha$ as a function of $\alpha$.  Since the one-stage cost $c \geq 0$ in (PC), for each policy $\pi$ and initial state $x$, the $\alpha$-discounted value $v^\pi_\alpha(x)$ is non-decreasing as $\alpha$ increases. Therefore, $m_\alpha = \inf_{x \in \X} v_\alpha(x)$ is also monotonically non-decreasing as $\alpha$ increases. It follows that $m_\alpha$ is a Borel measurable function of $\alpha$ on $(0,1)$ and so is $- m_\alpha$. Then, by \cite[Lem.~7.30(4)]{bs}, $h_\alpha(x) = v_\alpha(x) - m_\alpha$ is l.s.a.\ in $(\alpha, x)$. 
Now, since for each $\alpha$, $\uh_\alpha$ is the partial minimization of $h_\beta(x)$ over $\beta \in [\alpha, 1)$,  $\uh_\alpha$ is an l.s.a.\ function of $x$ by \cite[Prop.\ 7.47]{bs}. Finally, since $\uh_\alpha \uparrow \uh$ as $\alpha \uparrow 1$, we can write $\uh$ as the pointwise limit of $\uh_{\alpha_n}$ for a sequence $\alpha_n \uparrow 1$, and therefore, $\uh$ is l.s.a.\ by \cite[Lem.~7.30(2)]{bs}.
\end{proof}

\subsubsection{Majorization Condition} 
We now introduce a majorization condition for (PC). Its first two parts are the same as those of Assumption~\ref{cond-uc-2} for (UC).

\begin{assumption} \label{cond-pc-2}
In the model (PC), for each $x \in \X$ and $\epsilon > 0$, 
\begin{itemize}[leftmargin=0.8cm,labelwidth=!]
\item[\rm (i)] Assumption~\ref{cond-uc-2}(i) holds;
\item[\rm (ii)] Assumption~\ref{cond-uc-2}(ii) holds;
\item[\rm (iii)] Assumption~\ref{cond-uc-2}(iii) holds with $\uh$ in place of the weight function $w$:
\footnote{The integrals of $\uh$ appearing here are well-defined since $\uh$ is u.m.\ by Lemma~\ref{lem-h-lsa-pc}.}
\begin{equation}
   \lim_{\ell \to \infty} \sup_{a \in K_\epsilon(x) }  \int_\X \uh(y) \, \ind\big[\uh(y) \geq \ell \, \big] \, q(dy \mid x, a) = 0. \label{cond-pc-2c}
\end{equation}   
\end{itemize}
\end{assumption}

Recall that in Assumption~\ref{cond-uc-2}(i), for each state $x$ and $\epsilon > 0$, we need to choose a subset $K_\epsilon(x) \subset A(x)$ that contains $\epsilon$-optimal actions of the minimization problems associated with the right-hand sides of the $\alpha$-DCOEs for all $\alpha$ sufficiently large. For the (PC) model, if $c(x, \cdot)$ is unbounded above but the family $\{h_\alpha\}$ is pointwise bounded (i.e., (\ref{pc-cond-B}) holds), a strict subset of $A(x)$ can be found as a natural candidate for $K_\epsilon(x)$, without the need for detailed knowledge about $c(\cdot)$ or $\{h_\alpha\}$. For an illustrative example, see the example and the proof of Proposition~\ref{prp-expc-2} in Section~\ref{sec-ex-pc}.

About Assumption~\ref{cond-pc-2}(iii), in practice, to verify it without knowing $\uh$, one could first find an upper bound $\hat h \geq \uh$ and then verify the condition (\ref{cond-pc-2c}) for $\hat h$ instead:
$$  \lim_{\ell \to \infty} \sup_{a \in K_\epsilon(x) }  \int_\X \hat h(y) \, \ind\big[\hat h(y) \geq \ell \, \big] \, q(dy \mid x, a) = 0.$$
An upper bound $\hat h$ can be available as a byproduct of verifying the pointwise boundedness condition on the family $\{h_\alpha\}$ (either Assumption~\ref{cond-pc-1}(\underline{B}) or the stronger condition (\ref{pc-cond-B})). An illustrative example will be given in Section~\ref{sec-ex-pc}.

\subsubsection{Optimality Result} 
The following theorem can be compared with the existing ACOI results that require continuity/compactness conditions \cite{FKZ12,Sch93,VAm15}.

\begin{theorem}[the ACOI for (PC)] \label{thm-pc-acoi}
For the (PC) model, suppose Assumptions~\ref{cond-pc-1} and~\ref{cond-pc-2} hold. Let $\rho^* = \limsup_{\alpha \uparrow 1} (1 - \alpha) \, m_\alpha$ and let the real-valued function $\uh \in \An(\X)$ be as given in (\ref{eq-pc-uh}). Then the optimal average cost function $g^*(\cdot) = \rho^*$, and the pair $(\rho^*, \uh)$ satisfies the ACOI:
\begin{equation} \label{eq-pc-acoi}
  \rho^* + \uh(x) \geq \inf_{a \in A(x)} \left\{ c(x, a) + \int_\X \uh(y) \, q(dy \mid x, a) \right\}, \qquad x \in \X.
\end{equation}  
Hence there exist a nonrandomized Markov optimal policy and for each $\epsilon > 0$, a nonrandomized stationary $\epsilon$-optimal policy. 
\end{theorem}

\begin{rem}
The countable-state finite-action MDP in Cavazos-Cadena's counterexample for ACOE \cite{CCa91} satisfies both assumptions of Theorem~\ref{thm-pc-acoi}. In particular, regarding the majorization condition, in the finite-action case, Assumption~\ref{cond-pc-2}(i)-(ii) hold trivially for $K_\epsilon(x) = A(x)$; in that example $\uh = 0$, so Assumption~\ref{cond-pc-2}(iii) also holds. This shows that without further conditions, we cannot improve the ACOI result in Theorem~\ref{thm-pc-acoi} to the ACOE. \myqed
\end{rem} 

We now prove Theorem~\ref{thm-pc-acoi}. The proof is similar to that of the ACOI in the (UC) case, but there are some subtle differences, one of which being that in this case, for different states, we choose possibly different sequences $\alpha_n \uparrow 1$ and work with the corresponding relative value functions $h_{\alpha_n}$. We start with a lemma similar to Lemma~\ref{lem-uc-2acoi} for (UC); it follows from Egoroff's theorem and our majorization condition. 

\begin{lemma} \label{lem-pc-2acoi}
{\rm (PC)} Let Assumptions~\ref{cond-pc-1},~\ref{cond-pc-2}(ii)-(iii) hold, and
let $K_\epsilon(x) \subset A(x)$ be the set in the latter condition for a given $x \in \X$ and $\epsilon > 0$. Then for any sequence $\alpha_n \uparrow 1$, with $\uh_n = \uh_{\alpha_n}$, we have
\begin{align*}
& \lim_{n \to \infty} \inf_{a \in K_\epsilon(x)} \left\{ c(x, a) + \alpha_n \int_\X \uh_n(y) \, q(dy \mid x, a) \right\} 
 = \inf_{a \in K_\epsilon(x)} \left\{ c(x, a) +  \int_\X \uh(y) \, q(dy \mid x, a) \right\}.
\end{align*}
\end{lemma}

\begin{proof}
Note first that under Assumption~\ref{cond-pc-1}, $\uh_n \uparrow \uh < \infty$ as $n \to \infty$, and the integrals in the lemma are well-defined since $\uh_n, \uh$ are u.m.\ by Lemma~\ref{lem-h-lsa-pc}. We follow the proof for Lemma~\ref{lem-uc-2acoi} except for two small changes. The first change is that here, when dealing with the complement set $\X \setminus D_\delta$ after applying Egoroff's theorem, in order to obtain the bound (\ref{eq-prf-thm-uc1c}):
$$\sup_{a \in K_\epsilon(x)} \int_{\X \setminus D_\delta} \big(\uh(y) - \uh_n(y) \big)\, q(dy \mid x, a) \leq \eta \qquad \text{for all $n$ sufficiently large},$$ 
for a given $\eta > 0$ and sufficiently small $\delta > 0$, we use the fact that $\uh - \uh_n \leq \uh$ and we also use the uniform integrability condition in Assumption~\ref{cond-pc-2}(iii) for the nonnegative function $\uh$, instead of the weight function $w$. The second change is that near the end of the proof, to show that 
$$\textstyle{(1 - \alpha_n) \sup_{a \in K_\epsilon(x)}  \int_\X \uh(y) \, q(dy \mid x, a) \to 0} \ \ \ \text{as} \  n \to \infty,$$ 
we now use the fact that $\sup_{a \in K_\epsilon(x)}  \int_\X \uh(y) \, q(dy \mid x, a) < \infty$, which is again implied by the uniform integrability condition on $\uh$ in Assumption~\ref{cond-pc-2}(iii).
\end{proof}

\begin{proof}[Proof of Theorem~\ref{thm-pc-acoi}]
Consider an arbitrary $x \in \X$ and fix it in the proof below. Choose a sequence $\{\alpha_n\}$ such that 
$$\lim_{n \to \infty} h_{\alpha_n}(x) = \liminf_{\alpha \to 1} h_\alpha(x) = \uh(x).$$ 
Note that $\rho^* \geq \limsup_{n \to \infty} (1 - \alpha_n) \, m_{\alpha_n}$. Recall also that $\rho^* \leq \inf_{y \in \X} g^*(y)$ by~\cite[Lem.~1.2(b)]{Sch93}.
Define $h_n := h_{\alpha_n}$, $\uh_n := \uh_{\alpha_n}$. Then $\uh_n \leq h_n$ and $\uh_n \uparrow \uh$ as $n \to \infty$.

Let $\epsilon > 0$. For the fixed state $x$, by subtracting $\alpha m_\alpha$ from both sides of the $\alpha$-DCOE (Theorem~\ref{thm-dcoe}), we have that for all $n \geq 0$,
\begin{align}
   ( 1 - \alpha_n) \, m_{\alpha_n} + h_n(x) \, & = \, \inf_{a \in A(x)} \left\{ c(x, a) + \alpha_n \int_\X h_n(y) \, q(dy \mid x, a) \right\} \label{eq-prf-thm-pc-oe}.
\end{align}
Then similarly to the derivation of (\ref{eq-prf-thm-uc0a}) in the proof for Theorem~\ref{thm-uc-acoi}, we apply Assumption~\ref{cond-pc-2}(i) and replace $h_n$ by $\uh_n$ in the right-hand side above before letting $n \to \infty$ in both sides of (\ref{eq-prf-thm-pc-oe}). This results in the inequality
\begin{align}  \label{eq-prf-thm-pc-oe2}
    \rho^* + \uh(x)     
    & \, \geq \,   \liminf_{n \to \infty} \inf_{a \in K_\epsilon(x)} \left\{ c(x, a) + \alpha_n \int_\X \uh_n(y) \, q(dy \mid x, a) \right\} - \epsilon
\end{align}   
where $K_\epsilon(x)$ is the subset of actions in Assumption~\ref{cond-pc-2}(i) for the fixed state $x$ and $\epsilon > 0$. 
Next, by (\ref{eq-prf-thm-pc-oe2}) and Lemma~\ref{lem-pc-2acoi}, we have 
\begin{align*}
  \rho^* + \uh(x) + \epsilon & \, \geq \inf_{a \in K_\epsilon(x)} \left\{ c(x, a) + \int_\X \uh(y) \, q(dy \mid x, a) \right\} \\
  & \, \geq \, \inf_{a \in A(x)} \left\{ c(x, a) + \int_\X \uh(y) \, q(dy \mid x, a) \right\}.
\end{align*}  
Since this holds for an arbitrary $x \in \X$ and an arbitrary $\epsilon > 0$, we obtain the desired ACOI.
Since $\rho^* \leq g^*(\cdot)$~\cite[Lem.~1.2(b)]{Sch93}, the ACOI just established implies that we must have $g^*(\cdot) = \rho^*$ \cite[Prop.\ 1.3]{Sch93}. 
Finally, the existence of a Markov average-cost optimal policy and stationary $\epsilon$-optimal policy follows from Lemma~\ref{lem-h-lsa-pc}, the ACOI proved above, and Lemma~\ref{lem-optpol}. This completes the proof.
\end{proof}

\section{Further Discussion and Illustrative Examples} \label{sec-5}
In this section, we first explain the origin of the first part of our majorization condition and discuss an important implication of the second part of that condition (cf.\ Sections~\ref{sec-epilim}-\ref{sec-maj-wsc}). We will then discuss the boundedness condition (Assumption~\ref{cond-uc-1}) on the family of relative value functions in the (UC) case, a related $w$-geometric ergodicity condition, as well as stronger conditions for ensuring the ACOE (cf.\ Sections~\ref{sec-bd-uc}-\ref{sec-acoe-uc}). Finally, we will illustrate the application of our ACOI result for the (PC) model using an inventory control example (cf.\ Section~\ref{sec-ex-pc}); a similar example for (UC) is given in Appendix~\ref{appsec-ex-uc}.

\subsection{About Assumption~\ref{cond-uc-2}(i), Epi-limits of Functions, and Their Relation with ACOIs} \label{sec-epilim}

Assumption~\ref{cond-uc-2}(i) is the first part of our majorization condition for both (UC) and (PC) models. In introducing this assumption, we have been influenced by the theory on epi-convergence of functions~\cite[Chap.~7]{RoW98} and more specifically, by the relation between the epi-limits and pointwise limits of sequences of functions and certain inequalities for the infima of those functions and their limits.

If $\{f_n\}$ is a sequence of extended-real-valued functions on a metric space $Y$, the \emph{outer limit} of their epigraphs,
\footnote{The outer limit of a sequence $\{E_n\}$ of sets in a metric space is defined as $\limsup_{n \to \infty} E_n : = \cap_{m \geq 1} \cl \big( \cup_{n \geq m} E_n \big)$, where $\cl$ denotes the closure of a set. In our case, $E_n$ is the epigraph of $f_n$, $\epi(f_n) := \big\{ (y, r) \mid f_n(y) \leq r, \, y \in Y, \,  r \in \R\big\}$, and the set $\limsup_{n \to \infty} \epi(f_n)$ consists of all the points $(y, r) \in Y \times \R$ such that $r \geq \lim_{k \to \infty} f_{n_k}(y_k)$ for some subsequence $\{f_{n_k}\}$ of $\{f_n\}$ and some sequence of points $y_k \to y$.}
$\limsup_{n \to \infty} \epi(f_n)$, 
corresponds to the epigraph of some (extended-real-valued) function on $Y$. This function, denoted by $\elims_n f_n$, is called the \emph{lower epi-limit} of $\{f_n\}$. By definition, $\elims_n f_n$ is always l.s.c.\ and lies below the pointwise limit $\liminf_{n \to \infty}  f_n$, and moreover, 
\begin{equation} \label{ineq-epi1}
   \liminf_{n \to \infty} \inf_{y \in Y} f_n(y) \leq \inf_{y \in Y} \, \big({\textstyle{  \elims_n}} \,f_n \big)(y) \leq  \inf_{y \in Y} \liminf_{n \to \infty}  f_n(y).
\end{equation}   
Thus, if we want the equality $\liminf_{n \to \infty} \inf_{y \in Y} f_n(y) = \inf_{y \in Y} \liminf_{n \to \infty}  f_n(y)$, we will need $\liminf_{n \to \infty} \inf_{y \in Y} f_n(y) = \inf_{y \in Y} \big({\textstyle{\elims_n}} \,f_n \big)(y)$ to hold at least.

The proposition below gives a necessary, sometimes sufficient, condition for the latter equality. It is similar to \cite[Thm.~7.31(a)]{RoW98}; because \cite{RoW98} deals with finite-dimensional spaces only, for clarity, we include a proof of this result in Appendix~\ref{appsec-prf}.

\begin{prop} \label{prp-inf-epilim}
Let $\{f_n\}$ be a sequence of extended-real-valued functions on a metric space $Y$ such that $\liminf_{n \to \infty} \inf_{y \in Y} f_n(y)$ is finite. If 
\begin{equation} \label{epi-cond}
  \liminf_{n \to \infty} \inf_{y \in Y} f_n(y) = \inf_{y \in Y} \!\big({\textstyle{\elims_n}} \, f_n \big)(y),
\end{equation}  
then the following condition holds: for every $\epsilon > 0$, there exist a compact set $K \subset Y$ and a subsequence $\{f_{n_j}\}$ such that
\begin{equation} \label{epi-cond2}
   \inf_{y \in K} f_{n_j}(y) \leq \inf_{y \in Y} f_{n_j}(y) + \epsilon \qquad \text{for all $j$ sufficiently large}.
\end{equation}   
If $\lim_{n \to \infty} \inf_{y \in Y} f_n(y)$ exists and is finite, this condition implies (\ref{epi-cond}) .
\end{prop}

Let us discuss now what these results imply for the validity of the ACOI in MDPs. Consider some fixed state $\hat x \in \X$. Let Assumption~\ref{cond-uc-1} or~\ref{cond-pc-1} hold for (UC) or (PC), respectively. By redefining, if necessary, the sequence of discount factors $\alpha_n$ in the proofs of the ACOI in Section~\ref{sec-3}, we can suppose that the corresponding relative value functions $h_n := h_{\alpha_n}$ are such that $\lim_{n \to \infty} h_n(\hat x) = \uh(\hat x)$ and also that $\lim_{n \to \infty} (1- \alpha_n) \, m_{\alpha_n}$ exists if the model is (PC). Define functions $f_n$ on $A(\hat x)$ by
$$ f_n(a) = c(\hat x,a) + \alpha_n \int_\X h_n(y) \, q(dy \mid \hat x, a), \qquad a \in A(\hat x).$$ 
Then the following limit exists and is finite:
$$\lim_{n \to \infty} \inf_{a \in A(\hat x)}  f_n(a) = \rho(\hat x) + \uh(\hat x),$$
where $\rho(\hat x) = \rho^*$ for (UC) and $\rho(\hat x) = \lim_{n \to \infty} (1- \alpha_n) \, m_{\alpha_n} \leq \rho^*$ for (PC).
Applying Proposition~\ref{prp-inf-epilim} and (\ref{ineq-epi1}) to this sequence $\{f_n\}$ then gives the next corollary:

\begin{cor} \label{cor-epilim-acoi}
Let Assumption~\ref{cond-uc-1} (resp.~\ref{cond-pc-1}) hold for the (UC) (resp.\ (PC)) model, and consider the preceding $\{\alpha_n\}$, $\{h_n\}$, $\hat x$ and $\rho(\hat x)$. Suppose that for some $\epsilon > 0$, there does not exist a compact subset $K \subset A(\hat x)$ with the property that 
\begin{equation} \label{eq-diss-cond-K}
 \inf_{a \in K} \left\{ c(\hat x,a) + \alpha_{n_j} \int_\X v_{\alpha_{n_j}}(y) \, q(dy \mid \hat x, a) \right\} \leq v_{\alpha_{n_j}}(\hat x) + \epsilon
\end{equation}
for some subsequence $\{\alpha_{n_j}\}$ and all $j$ sufficiently large. Then 
\begin{equation} \label{eq-diss-acoi}
   \rho(\hat x) + \uh(\hat x) < \inf_{a \in A(\hat x)}  \left\{ c(\hat x,a) + \liminf_{n \to \infty} \int_\X h_n(y) \,  \, q(dy \mid \hat x, a) \right\}.
\end{equation}
\end{cor}

While the strict inequality (\ref{eq-diss-acoi}) does not contradict the ACOI in general, it rules out the ACOI in the following situations, which could happen in a given problem: $h_n \uparrow \uh$ and, if the model is (PC), $\rho(\hat x) = \rho^*$ in addition.
In these cases, on the right-hand side of (\ref{eq-diss-acoi}), we have $\liminf_{n \to \infty} \int_\X h_n(y) \,  q(dy \mid \hat x, a) = \int_\X \uh(y) \, q(dy \mid \hat x, a)$ for each $a$ by the monotone convergence theorem, so (\ref{eq-diss-acoi}) becomes
$$ \rho^* + \uh(\hat x) < \inf_{a \in A(\hat x)}  \left\{ c(\hat x,a) + \int_\X \uh(y) \, q(dy \mid \hat x, a) \right\},$$
violating the ACOI at the state $\hat x$.
It is this observation that led us to introduce a condition like the inequality (\ref{eq-diss-cond-K}) as part of our majorization condition for the ACOI. The inequality (\ref{eq-diss-cond-K}) itself is unwieldy to verify, since, among others, it depends on the particular choice of the sequence $\{\alpha_n\}$. So we required a similar inequality to hold for all sufficiently large $\alpha$ instead, for each state $x$, and this gave rise to Assumption~\ref{cond-uc-2}(i). (In the latter, the subset of actions is not required to be compact, for that is not needed in our proof of the ACOI.)

Finally, note that Proposition~\ref{prp-inf-epilim} only deals with the first relation in (\ref{ineq-epi1}) and does not guarantee equality to hold throughout (\ref{ineq-epi1}), for the second inequality in (\ref{ineq-epi1}) can still be strict. Likewise, Assumption~\ref{cond-uc-2}(i) alone (even with a nontrivial, compact strict subset $K_\epsilon(x)$ of $A(x)$) is in general not sufficient for the ACOI to hold. It is the two other parts of Assumption~\ref{cond-uc-2} that gave us the rest of the help needed.

\subsection{About Assumption~\ref{cond-uc-2}(ii), Weak Sequential Compactness, and Alternative Proof of Lemmas~\ref{lem-uc-2acoi}, \ref{lem-pc-2acoi}} \label{sec-maj-wsc}

The existence of majorizing finite measures required by (\ref{cond-uc-2b}) in Assumption~\ref{cond-uc-2}(ii) has the following important implication. For each state $x$, the set of probability measures, $\Q_\epsilon(x): = \{ q(dy \,|\, x, a) \mid a \in K_\epsilon(x)\}$, viewed as a subset in the Banach space of finite signed Borel measures on $\X$ (endowed with the total variation norm), must be \emph{weakly sequentially compact}, i.e., sequentially compact
\footnote{Here the notion of sequential compactness is as defined by Dunford and Schwartz \cite[Def.~I.6.10]{DuS57}.}  
w.r.t.\ the weak topology on the latter space induced by its topological dual \cite[Prop.~1.4.4, Cor.~1.4.5]{HL03}. Hence, any sequence in $\Q_\epsilon(x)$ has a subsequence that converges \emph{setwise} to some probability measure in $\P(\X)$ (cf.\ \cite[Chaps.~1.4.1, 1.5.2]{HL03}). 

This gives us an alternative proof of Lemmas~\ref{lem-uc-2acoi},~\ref{lem-pc-2acoi}. Let us sketch it for Lemma~\ref{lem-uc-2acoi} here (the arguments for these two lemmas are similar). As in the previous proof of Lemma~\ref{lem-uc-2acoi}, by Assumptions~\ref{cond-uc-1},~\ref{cond-uc-2}(iii), and the fact $\uh_n \leq \uh$, to prove the lemma, it suffices to prove the inequality
\begin{align*}
 \ell_1  : =  &  \liminf_{n \to \infty} \inf_{a \in K_\epsilon(x)} \left\{ c(x, a) + \int_\X \uh_n(y) \, q(dy \mid x, a) \right\} 
 \geq 
    \inf_{a \in K_\epsilon(x)} \left\{ c(x, a) +  \int_\X \uh(y) \, q(dy \mid x, a) \right\} =: \ell_2.
\end{align*}  
Now, take a sequence $\{a_n\} \subset K_\epsilon(x)$ and choose a subsequence $\{a_{n_k}\}$ from it such that (i) $\ell_1$ equals
$$ \liminf_{n \to \infty} \left\{ \!c(x, a_n) + \!\int_\X \uh_n(y) \, q(dy \mid x, a_n) \! \right\} 
 \! = \! \lim_{k \to \infty} \!\left\{ \!c(x, a_{n_k}) + \!\int_\X \uh_{n_k}(y) \, q(dy \mid x, a_{n_k}) \! \right\}\!,$$
and (ii)
$q(dy \,|\, x, a_{n_k})$ converges setwise to some probability measure $\bar p \in \P(\X)$ 
as $k \to \infty$.
Using the setwise convergence property (ii), the fact $\uh_n \uparrow \uh$, and the uniform integrability condition in Assumption~\ref{cond-uc-2}(iii), it can be shown
\footnote{To prove this, one can use arguments from \cite{FKL19,HLL00-b} or \cite[p.~231]{Roy68} about limit theorems for integrals involving setwise convergent measures.}
that as $k \to \infty$,
$$\int_\X \uh_{n_k}(y) \, q(dy \,|\, x, a_{n_k}) \to \int \uh \, d \bar p, \qquad \int_\X \uh(y) \, q(dy \,|\, x, a_{n_k}) \to \int \uh \, d \bar p.$$
This implies
$\ell_1 = \lim_{k \to \infty} \left\{ c(x, a_{n_k}) + \int_\X \uh(y) \, q(dy \,|\, x, a_{n_k}) \right\} \geq \ell_2$, the desired inequality to establish Lemma~\ref{lem-uc-2acoi}.

Note that the limit measure $\bar p$ in the above need not lie in $\{ q(dy \,|\, x, a) \mid a \in A(x) \}$; even if the subsequence $a_{n_k} \to \bar a \in A(x)$, it does not imply $q(dy\,|\, x, \bar a) = \bar p$. In other words, the weak sequential compactness of $\Q_\epsilon(x)$ entailed by the majorization condition in Assumption~\ref{cond-uc-2}(ii) is different from the condition of $q(dy \,|\, x, a)$ being setwise continuous in $a$ (for each $x$) considered by \cite{HLe91,Sch93}. 

The above discussion also suggests two possible directions to improve the results of this paper. One is to exploit the full potential of Egoroff's theorem, and the other is to use a measure-theoretic approach instead of Egoroff's theorem, under weaker conditions than Assumption~\ref{cond-uc-2}(ii).

Finally, we note that previously, \cite{GhH99,HeR01} applied majorization type conditions and the weak sequential compactness property to completely or partially observable, discounted-cost MDPs. However, their focuses are on the existence of nonrandomized stationary optimal policies (with the initial distribution being fixed in \cite{GhH99}) for l.s.c.\ MDP models.

\subsection{About the Assumption $\sup_{\alpha \in (0,1)} \|h_\alpha\|_w < \infty$ in the (UC) Case} \label{sec-bd-uc}

This is Assumption~\ref{cond-uc-1}. 
It holds if the MDP is \emph{uniformly $w$-geometrically ergodic},
\footnote{This means that every stationary nonrandomized policy $\mu$ induces a Markov chain on $\X$ with a unique invariant probability measure $\varphi_\mu$, and the $n$-step transition kernel $q_\mu^n(dy \,|\, x)$ of this Markov chain satisfies that $\| q^n_\mu - \varphi_\mu \|_w \leq R \, r^n$ for all $n \geq 1$, where
$$ \| q^n_\mu - \varphi_\mu \|_w : = \sup_{x \in \X} \, \sup_{| f| \leq w, \, f \in \M(\X)} w(x)^{-1} \cdot \Big| \int_\X f(y) \, q_\mu^n(dy \,|\, x) - \int_\X f(y) \, \varphi_\mu(dy) \Big|, $$
and $R > 0$, $0 < r < 1$ are constants \emph{independent of the policy $\mu$} (cf.~\cite[Assumption 10.2.2]{HL99}). If $w(\cdot) \equiv 1$, the MDP is called \emph{uniformly geometrically ergodic}.
\label{footnote-ergMDP}
} 
as shown by \cite[Lem.~10.4.2]{HL99}---this result applies in our case as well, in view of Theorem~\ref{thm-dcoe} on the existence of stationary $\epsilon$-optimal policies in the $\alpha$-discounted problems. Geometrically or $w$-geometrically ergodic Markov chains and their application to MDPs are well-studied; see the book \cite{HL99} for an in-depth discussion and a large body of the related literature.

Assumption~\ref{cond-uc-1} can also hold in MDPs that do not possess ergodicity properties. A simple class of such MDPs is the ``invariant'' model considered by Assaf \cite{Ass80}: $A(x) =\A$ for all $x \in \X$ and the state transition stochastic kernel is $q(dy \,|\, a)$, independent of $x$. In this case, if $|c(x,a)|  \leq \hat c \, w(x)$ for all $(x,a) \in \Gamma$ and a weight function $w(\cdot)$, then
\begin{equation} \label{eq-invmod1}
 | v_\alpha(x) - v_\alpha(\bar x) | \leq \textstyle{ \sup_{a \in \A}} | c(x, a) - c(\bar x, a) | \leq \hat c \big( w(x) + w(\bar x) \big) \quad \forall \, \alpha \in (0,1),
\end{equation} 
and hence $\sup_{\alpha \in (0,1)} \|h_\alpha\|_w < \infty$. In the special case of bounded one-stage costs, the family $\{h_\alpha\}$ is uniformly bounded.

For (UC) with an unbounded weight function $w(\cdot)$, Assumption~\ref{cond-uc-1} also holds for ``partially invariant'' models given in the next example.

\begin{example}[partially invariant models] \label{ex-pt-invmod}
Consider a slight extension of Assaf's invariant models \cite{Ass80} in the case of (UC). Let $b, \hat c \geq 0$ and $\lambda \in [0,1)$ be the constants in the model conditions of (UC), and let the weight function $w(\cdot)$ be l.s.a.
Suppose that for some $\lambda' \in (\lambda, 1)$, the MDP model is ``invariant'' on the subset of states
$$\hat{\X} : = \big\{ x \in \X \mid w(x) \leq b/(\lambda' - \lambda) \big\},$$
in the sense that for all $x \in \hat{\X}$, $\hat A : = A(x)$ is the same and $q(dy \,|\, x,a)$ depends only on $a$. 
Then (\ref{eq-invmod1}) holds for $x, \bar x \in \hat \X$ with $\hat A$ in place of $\A$, so
\begin{equation} \label{eq-invmod1b}
 \textstyle{ \big| v_\alpha(x) - v_\alpha(\bar x) \big| \leq \hat c \big( w(x) + w(\bar x) \big) \leq 2 b \hat c /(\lambda' - \lambda)} \quad \forall \, x, \bar x \in \hat \X.
\end{equation} 

Fix $\bar x \in \hat \X$. Let us bound $v_\alpha(x) - v_\alpha(\bar x)$ for $x \not\in \hat \X$. 
Let $\tau = \inf \{ n \geq 0 \mid x_n \in \hat \X \}$, the first entrance time to $\hat \X$. 
By the (UC) model condition (b) (cf.\ Def.~\ref{def-uc}) and the definition of the set $\hat \X$, we have that 
\begin{equation} \label{eq-invmod2a}
   \int_\X w(y) \, q( dy \mid x, a) \leq \lambda' w(x) \qquad \forall \, x \not\in \hat \X, \, a \in A(x).
\end{equation}   
It implies that under any policy $\pi$,
\begin{equation} \label{eq-invmod2}
 \E_x^\pi \big[ \textstyle{\sum}_{n=0}^{\tau-1} \, w(x_n) \big] \leq \ell w(x)  \qquad \forall \, x \not\in \hat \X,
\end{equation} 
where $\ell$ is some constant independent of $\pi$ and $x$. This inequality follows from a general comparison theorem \cite[Thm.~15.2.5] {MeT09} and can also be verified directly (in fact, $\ell = (1 - \lambda')^{-1}$ in this case). 
For each $\alpha \in (0,1)$, let $\mu_\alpha$ be a stationary $\epsilon$-optimal policy for the $\alpha$-discounted problem (cf.\ Theorem~\ref{thm-dcoe}). Using the inequality (\ref{eq-invmod2}) together with the (UC) model condition (a) and the strong Markov property,  we have
\begin{align}
    v_\alpha(x) & \leq  \E_x^{\mu_\alpha} \big[ \textstyle{\sum}_{n=0}^{\tau-1} \, \alpha^n c(x_n, a_n)  +  \alpha^\tau v_\alpha(x_{\tau}) \big] \notag \\
      & \leq \hat c \, \ell w(x) + v_\alpha(\bar x) + \E_x^{\mu_\alpha} \big[ \alpha^\tau v_\alpha(x_{\tau}) - v_\alpha(\bar x) \big],  \label{ineq-part-invmod1}\\    
      v_\alpha(x) & \geq  \E_x^{\mu_\alpha} \big[ \textstyle{\sum}_{n=0}^{\tau-1} \, \alpha^n c(x_n, a_n)  +  \alpha^\tau v_\alpha(x_{\tau}) \big] -  \epsilon \notag \\
      & \geq - \hat c \, \ell w(x) + v_\alpha(\bar x) +  \E_x^{\mu_\alpha} \big[ \alpha^\tau v_\alpha(x_{\tau}) - v_\alpha(\bar x) \big]  -  \epsilon. \label{ineq-part-invmod2}
\end{align}
Now 
\begin{align}
  \big| \E_x^{\mu_\alpha} \big[ \alpha^\tau v_\alpha(x_{\tau}) - v_\alpha(\bar x) \big] \big| & = \big|  \E_x^{\mu_\alpha} \big[ \alpha^\tau \big(v_\alpha(x_{\tau}) - v_\alpha(\bar x) \big) + (\alpha^\tau - 1) v_\alpha(\bar x) \big] \big| \notag \\
  & \leq \textstyle{ \sup_{x' \in \hat \X} \big| v_\alpha(x') - v_\alpha(\bar x) \big| + \E_x^{\mu_\alpha} [ 1 - \alpha^\tau ] \, \big| v_\alpha(\bar x) \big|.} \label{eq-invmod3}
\end{align}  
The first term in (\ref{eq-invmod3}) is bounded by a constant by (\ref{eq-invmod1b}).
To bound the second term, note that $\E_x^{\mu_\alpha} \big[ 1 - \alpha^\tau \big] \leq (1 - \alpha) \, \E_x^{\mu_\alpha} [ \tau ]$, and  $\E_x^{\mu_\alpha} [ \tau ] \leq \ell w(x)$ by (\ref{eq-invmod2}), whereas due to the (UC) model condition, $(1 - \alpha) |v_\alpha(\bar x)|$ is bounded by some constant $\ell_{\bar x}$ independent of $\alpha$. Therefore, the second term in (\ref{eq-invmod3}) is bounded by $ \ell \ell_{\bar x} w(x)$. We thus have
\begin{equation}  \label{eq-invmod3b}
\big| \E_x^{\mu_\alpha} \big[ \alpha^\tau v_\alpha(x_{\tau}) - v_\alpha(\bar x) \big] \big| \leq 2 b \hat c /(\lambda' - \lambda) +  \ell \ell_{\bar x} w(x).
\end{equation}
Combining this relation with the upper and lower bounds on $v_\alpha(x)$ given in (\ref{ineq-part-invmod1})-(\ref{ineq-part-invmod2}) and taking $\epsilon \to 0$, we obtain that for all $\alpha \in (0,1)$, 
\begin{align*}
       - \tilde \ell_{\bar x} w(x) - \tilde \ell  \leq v_\alpha(x) - v_\alpha(\bar x) & \leq \tilde \ell_{\bar x} w(x) + \tilde \ell 
\end{align*} 
where $\tilde \ell_{\bar x} =  \hat c \, \ell + \ell \ell_{\bar x}$ and $\tilde \ell = 2 b \hat c /(\lambda' - \lambda)$ are constants independent of $\alpha$ and $x$.
This shows that $\sup_{\alpha \in (0,1)} \|h_\alpha\|_w < \infty$.
\myqed
\end{example}

\begin{rem}[ACOI for invariant and partially invariant models]
For invariant models, Assaf \cite{Ass80} proved the ACOE when the one-stage costs are bounded and the action space is finite. Sch{\"a}l \cite[sect.~4, p.~169]{Sch93} showed that, for a Borel action space, if $c$ is bounded below with $\sup_{a \in \A} c(x, a) < \infty$ for all $x \in \X$, then $\sup_{\alpha \in (0,1)} h_\alpha(x) < \infty$ for all $x \in \X$. Thus, in that case, the ACOI holds under our majorization condition for (PC), by Theorem~\ref{thm-pc-acoi}.
In the (UC) case, by Theorem~\ref{thm-uc-acoi}, the ACOI also holds under our majorization condition, for both the invariant models and the  partially invariant models described in the preceding example.
\myqed
\end{rem}

We will give an illustrative example in Appendix~\ref{appsec-ex-uc} to demonstrate another way to ensure that Assumption~\ref{cond-uc-1} holds in an MDP that is not necessarily ergodic. There we will use some arguments from Example~\ref{ex-pt-invmod}. In general, however, verifying Assumption~\ref{cond-uc-1} is a hard problem since it involves the optimal value functions $\{v_\alpha\}$.

\subsection{About Geometric Ergodicity Conditions, ACOI, and ACOE in the (UC) Case} \label{sec-acoe-uc}

For a uniformly $w$-geometrically ergodic MDP (cf.\ footnote~\ref{footnote-ergMDP} in Section~\ref{sec-bd-uc}), under compactness/continuity and additional conditions, the ACOI has been strengthened to the ACOE; see e.g., Gordienko and Hern\'{a}ndez-Lerma~\cite[Thm.~2.8]{GHL95}, Hern\'{a}ndez-Lerma and Vega-Amaya \cite[Thm.~3.5(a)]{HLV98}, Ja\'{s}kiewicz and Nowak~\cite[Thm.~3]{JaN06}, and also the book \cite[Thm.~10.3.6]{HL99}. 
As recounted in Section~\ref{sec-1}, the ACOE was also directly obtained through a contraction-based, fixed point approach in Vega-Amaya~\cite{VAm03,VAm18} for MDPs under certain compactness and continuity conditions, and in Ja\'{s}kiewicz~\cite{Jas09} for SMDPs with u.m.\ policies.
\footnote{One difference between the minorization conditions in \cite{Jas09} and \cite{VAm03,VAm18} is that the former is stated in terms of ``small sets,'' whereas the latter ``small functions'' (cf.\ \cite[Chap.\ 2.3]{Num84}).}
(In the special case $w(\cdot) \equiv 1$, the ACOE was obtained earlier, through the fixed point approach, by Gubenko and Shtatland~\cite{GuS75}, as mentioned in the introduction.)
 
For these ACOE results to hold in an MDP, among others, the MDP must be uniformly $w$-geometrically ergodic,
\footnote{In~\cite{VAm03,VAm18}, this follows as an implication of the conditions involved, by \cite[Thm.~3.3]{VAm03} together with \cite[Thm.~6]{Kar85} (see also \cite[Thm.~7.3.11]{HL99}).}
and moreover, there must exist a nontrivial $\sigma$-finite measure $\varphi$ on $\X$ such that under any stationary nonrandomized policy $\mu$, the induced Markov chain $\{x_n\}$ is \emph{$\varphi$-irreducible} (see e.g.,~\cite{HL99,MeT09,Num84} for definition). The existence of such a common irreducibility measure appears either in the conditions of the aforementioned results (e.g., \cite[Thm.~10.3.6]{HL99}) or follows as an implication of their assumptions (e.g., \cite{Jas09,VAm03}). 

The uniform $w$-geometric ergodicity condition implies that for any \emph{two} stationary nonrandomized policies, there must be a common irreducibility measure, but a common irreducibility measure for \emph{all} stationary nonrandomized policies need not exist. This is illustrated by the following counterexample.

\begin{example}[non-existence of common irreducibility measure] \rm 
Let $\X$ be the unit circle and $\A = \{0,1,2\} = A(x)$ for all $x \in \X$. Represent states by their angular coordinates. Let $q(dy \,|\, x, a)$ be the uniform distribution on $[\tfrac{2a}{3} \pi, \, \tfrac{2a}{3} \pi + \pi] = : I_a$ for $a \in \A$ and every $x$. Then the MDP is uniformly geometrically ergodic, and the \emph{maximal irreducibility measures} (see~\cite{MeT09,Num84} for definition) that are possible under any nonrandomized stationary policy, are (up to equivalence) the Lebesgue measure restricted on the intervals $I_a$, $a = 0, 1, 2$, on the unions of any two such intervals, and on $[0, 2 \pi]$. Since the intersection of these intervals is empty, there is no common irreducibility measure for all stationary nonrandomized policies in this MDP. \myqed
\end{example}

This example shows that the uniform $w$-geometric ergodicity condition is more general than the conditions required by the fixed point approach studied in \cite{Jas09,VAm03,VAm18}, whereas Assumption~\ref{cond-uc-1} is more general than the former condition, as discussed in Section~\ref{sec-bd-uc}. 
This suggests that most likely, the ACOI in our Theorem~\ref{thm-uc-acoi} cannot be strengthened to ACOE without additional assumptions. However, we do not know a counterexample for ACOE under the conditions of Theorem~\ref{thm-uc-acoi}. 

\subsection{An Illustrative Example for the (PC) Case} \label{sec-ex-pc}

In this subsection, we use an inventory control example to illustrate the application of our ACOI result for the (PC) case.
This example is adapted from an example in \cite[p.~170]{Sch93}. The original example has compact action sets $A(x)$ and i.i.d.\ random demands. We make $A(x)$ non-compact and also allow action-dependent random demands, with which the state transition stochastic kernel need not be continuous. We will verify two major conditions for applying Theorem~\ref{thm-pc-acoi}: our majorization condition and the condition (B), $\sup_{\alpha \in (0,1)} h_\alpha(x) < \infty$, $x \in \X$, which implies Assumption~\ref{cond-pc-1}(\underline{B}). 

\subsubsection{Two Helpful Lemmas}
The condition (B) is not easy to check because it involves the optimal value functions $v_\alpha$, whose structures are problem-dependent and hard to characterize in general. Most of our efforts will be on verifying this condition. We will need two helpful lemmas given below. They are from \cite[Lems.~4.4,~4.6]{Sch93}, with slight modifications in the first one. Their proof arguments are outlined in \cite{Sch93}, which utilize various stopping times. For completeness, we include a proof for the second lemma (Lemma~\ref{lem-Sch2} below) in Appendix~\ref{appsec-prf}, as it is less obvious.

Let $\uc(x) : = \inf_{a \in A(x)} c(x,a)$ and $\underline{g}^* : = \inf_{x \in \X} g^*(x)$. Note that $\uc(\cdot)$ is l.s.a.~\cite[Prop.~7.47]{bs} and $\underline{g}^* < \infty$ under Assumption~\ref{cond-pc-1}(G). Define also
$$ \textstyle{   \uv_\alpha (x) : = \inf_{a \in A(x)} \int_\X v_\alpha(y) \, q(dy \mid x, a), \qquad \um_\alpha : = \inf_{x \in \X} \uv_\alpha(x).}$$
Note that $\alpha \um_\alpha \leq m_\alpha \leq \um_\alpha$, the function $\uv_\alpha$ is l.s.a.~\cite[Props.~7.47,~7.48]{bs}, and by \cite[Prop.~7.50]{bs}, for any $\epsilon > 0$, there is a (u.m.)~nonrandomized stationary policy $\mu^\epsilon_\alpha$ with
\begin{equation} \label{eq-mu-al-ep}
   \textstyle{\int_\X v_\alpha(y) \, q(dy \mid x, \mu^\epsilon_\alpha(x)) \leq \uv_\alpha(x) + \epsilon \qquad \forall \, x \in \X.}
\end{equation}   
Denote by $\F_n$ the $\sigma$-algebra generated by $(x_0, a_0, \ldots, x_n, a_n)$.

\begin{lemma}[cf.\ {\cite[Lem.~4.4]{Sch93}}] \label{lem-Sch1}
Let $\eta, \epsilon > 0$. Let $\pi$ be any policy and $\tau$ be any stopping time w.r.t.\ $\{\F_n\}$ such that $ \alpha \, \uv_\alpha(x_\tau) \leq m_\alpha + \eta$ on $\{ \tau < \infty \}$. 
Then 
$$ h_\alpha(x) \leq \eta + \epsilon + \E^\pi_x \left[ \, \sum_{n=0}^{\tau -1} c(x_n, a_n) + c\big(x_{\tau}, \mu^\epsilon_\alpha(x_\tau)\big) \right] \qquad \forall \, x \in \X. $$
\end{lemma}

\begin{lemma}[cf.\ {\cite[Lem.~4.6]{Sch93}}] \label{lem-Sch2}
Let Assumption~\ref{cond-pc-1}(G) hold. Let $\epsilon > 0$ and $G \subset \Gamma$ be a u.m.\ set such that $\int_\X \uc(y) \, q(dy \, |\, x,a) \geq \underline{g}^* + \epsilon$ for all $(x,a) \in G$. Then, for some $\alpha_\epsilon < 1$, it holds for all $\alpha \geq \alpha_\epsilon$ that
$$   \int_\X v_\alpha(y) \, q(dy \mid x, a) \geq \um_\alpha + \epsilon/2 \qquad \forall \, (x, a) \in G.$$
\end{lemma}

\subsubsection{Single-Product Inventory System} \label{sec-pc-ex}

Let $\X = \A = \R$, $A(x) = [x, +\infty)$, and $c(x, a) = \kappa(a-x) + \psi(a)$ where $\kappa: \R_+ \to \R_+$ and $\psi: \R \to \R_+$.
Let the states evolve according to $x_{n+1} = a_n - \xi_n$, where, for the $n$th stage:
\begin{itemize}[leftmargin=0.5cm,labelwidth=!]
\item $x_n$ and $a_n$ are the stock levels before and after placing the order, respectively;
\item $\xi_n \geq 0$ is the demand or consumption of the product, which is a random variable that does not depend on the history $(x_0, a_0, \xi_0, \ldots, x_n)$ given $a_n$;
\item $\kappa(a_n - x_n)$ and $\psi(a_n)$ correspond to the ordering cost and the holding or shortage/backordering cost, respectively.
\end{itemize}
We assume that for all $n \geq 0$, the conditional distributions of $\xi_n$ given $a_n$ are parametrized by the values of $a_n$ (and Borel measurable in $a_n$), independent of $n$. These probability distributions will be denoted by $F_a$, $a \in \R$. In what follows, we shall write $\xi_n(a_n)$ for $\xi_n$ to emphasize the dependence on $a_n$, and we use $\xi(a)$ to denote a generic random variable with probability distribution $F_a$ and $\E [ \xi(a) ]$ to denote its expectation w.r.t.\ $F_a$.

We impose two sets of conditions on the functions $\kappa(\cdot), \psi(\cdot)$ and the demand distributions $F_a$.
The first set is given below and will be used shortly to prove that the family $\{h_\alpha\}$ is pointwise bounded under those conditions. When we consider the majorization condition later, we will add a second set of conditions. 

Throughout this example, let Assumption~\ref{cond-pc-1}(G) hold
\footnote{In fact, (G) holds under Assumption~\ref{ex-storage-cond1}(i) and (iv): it can be easily verified that $J(\pi, 0) < \infty$ for the nonrandomized stationary policy $\pi$ that has $\pi(x) = 0$ for $x < 0$ and $\pi(x) = x$ for $x \geq 0$.}
(hence $\underline{g}^* < \infty$).

\begin{assumption}[conditions on $\kappa(\cdot), \psi(\cdot)$, and $\{\xi(a)\}$] \label{ex-storage-cond1} \hfill 
\begin{enumerate}[leftmargin=0.6cm,labelwidth=!]
\item[\rm (i)] The ordering cost $\kappa(\cdot)$ and the holding or shortage cost $\psi(\cdot)$ are bounded above on bounded intervals; moreover, $\kappa(\cdot)$ is non-decreasing and $\kappa(0) = 0$.
\item[\rm (ii)] $\liminf_{z \to \infty} \kappa(z) > \underline{g}^*$ and $\liminf_{|a| \to \infty} \psi(a) > \underline{g}^*$.  
\item[\rm (iii)] For any $\bar a \in \R$, $\{\xi(a) \, | \, a \geq \bar a\}$ is uniformly integrable.
\item[\rm (iv)] For any bounded interval $I \subset \R$ and any given $z \geq 0$, 
$$\textstyle{\inf_{a \in I} \E[\xi(a)] > 0 \ \ \ \text{and} \ \ \ \sup_{a \in I} \E \big[ \kappa \big( z + \xi(a) \big) \big] < \infty.}$$
\end{enumerate}
\end{assumption}

Note that here $\kappa(\cdot)$ and $\psi(\cdot)$ need not be l.s.c., and $q(dy\,|\,x,a)$ need not be continuous. As a simple example, these functions can be piecewise continuous; even in this case, the compactness/continuity conditions required in \cite{CoD12,FKZ12,Sch93} are violated.

\begin{prop} \label{ex-storage-prp1}
Under Assumption~\ref{ex-storage-cond1}, for all $x \in \R$, $\sup_{\alpha \in (0,1)} h_\alpha(x) < \infty$. 
\end{prop}

\begin{proof}
We follow the line of reasoning in~\cite[sect.~4]{Sch93}. When computing upper bounds on $h_\alpha(x)$, however, we will need to use more complex arguments because the random demands in our case are not i.i.d.\ as in \cite{Sch93}. 

Let us write $q(dy \,|\, a)$ for $q(dy \,|\, x, a)$, since state transitions in this example depend on $a$ only. First, we use Lemma~\ref{lem-Sch2} to assert the existence of a bounded interval $[L, M]$ such that for some $\bar \epsilon > 0$ and $\alpha_{\bar \epsilon} < 1$, it holds \emph{for all $\alpha \geq \alpha_{\bar \epsilon}$} that
\begin{equation} \label{exst-ya}
 \exists \, y_\alpha \in [L, M] \ \ \text{with} \ \ \int_\X v_\alpha(y) \, q(dy \mid y_\alpha) \leq \um_\alpha + \bar \epsilon/2.
\end{equation} 
(This is proved by a direct calculation of a set $G$ in the condition of Lemma~\ref{lem-Sch2} using the definition of the function $\uc$ and Assumption~\ref{ex-storage-cond1}(ii)-(iii).) The existence of such $y_\alpha$ allows us to define a stopping time $\tau$ needed in applying Lemma~\ref{lem-Sch1} to bound $h_\alpha$: since for all $x' \leq y_\alpha$, we have $y_\alpha \in A(x')$ and
\begin{equation} \label{ineq-exst-1}
 \alpha \, \uv_\alpha(x') \leq \alpha \int_\X v_\alpha(y) \, q(dy \mid y_\alpha) \leq \alpha \, \um_\alpha + \bar \epsilon/2 \leq m_\alpha + \bar \epsilon/2,
\end{equation}
stopping at any state $x' \leq y_\alpha$ will satisfy the condition of Lemma~\ref{lem-Sch1} on $\tau$. 
This is the basic idea to bound $\sup_{\alpha \in (0,1)} h_\alpha(x)$. We now proceed to calculate the bounds on $h_\alpha(x)$ for each $x \in \R$ and an arbitrary $\alpha \geq \alpha_{\bar \epsilon}$.

Define a stationary policy $\mu$ by 
$$ \mu(x) = x \ \  \ \text{if} \ \ x \geq L, \qquad \mu(x) = y_\alpha \ \ \ \text{if} \ \ x < L.$$
Let $\tau : = \inf \{ n \geq 0 \mid x_n < L \}$. 
We will apply Lemma~\ref{lem-Sch1} with this $\tau$ and $\pi=\mu$. To this end, let us bound $\E_x^\mu \left[ \, \sum_{n=0}^{\tau -1} c(x_n, a_n) + c(x_{\tau}, y_\alpha) \right]$ independently of $\alpha$. The case $x < L$ is simple: we have $\tau = 0$ and under Assumption~\ref{ex-storage-cond1}(i),
\begin{equation}
c(x, y_\alpha) = \kappa(y_\alpha - x) + \psi(y_\alpha) \leq  \kappa(M-x) + \textstyle{\sup_{y \in [L, M]}} \, \psi(y) < \infty. \label{exst-ineq1b} 
\end{equation} 
For $x \geq L$, since $\kappa(0)=0$ by Assumption~\ref{ex-storage-cond1}(i), we have
\begin{align}
 \E_x^\mu \left[ \, \sum_{n=0}^{\tau -1} c(x_n, a_n) + c(x_{\tau}, y_\alpha) \right] & = \E_x^\mu \left[ \, \sum_{n=0}^{\tau -1} \psi(x_n) + \psi(y_\alpha) + \kappa(y_\alpha - x_{\tau})  \right] \notag \\
 & \leq \E_x^\mu [\tau] \cdot \sup_{y \in [L, \, x \vee M]} \psi(y)  + \E_x^\mu \big[ \kappa(y_\alpha - x_{\tau}) \big] \label{exst-ineq2}
\end{align} 
where $x \vee M : = \max \{ x, \, M\}$. So we need to bound $\E_x^\mu [\tau]$ and $\E_x^\mu \big[ \kappa(y_\alpha - x_{\tau}) \big]$.

We apply a comparison theorem \cite[Prop.\ 11.3.2]{MeT09} together with other arguments to bound $\E_x^\mu [\tau]$, which in turn yields a bound on $\E_x^\mu \big[ \kappa(y_\alpha - x_{\tau}) \big]$. The derivations are given in Appendix~\ref{appsec-derivations-ex-pc}. The obtained upper bounds are as follows: for $x \geq L$,
\begin{align}
   \E_x^\mu \big[ \kappa (y_\alpha - x_{\tau}) \big] &  \leq \E_x^\mu [\tau]  \cdot \sup_{y \in [L, \, x \vee M]} \E \Big[ \kappa \big(M - L + \xi(y)\big) \Big], \label{exst-ineq4} \\
  \E_x^\mu [\tau] &  \leq 2 (x - L + D_x)/\Delta_x,  \label{exst-ineq3}
\end{align} 
where $\Delta_x = \inf_{y \in [L, \, x \vee M]} \E \big[ \xi(y) \big] > 0$ (cf.\ Assumption~\ref{ex-storage-cond1}(iv)), and $0 < D_x < \infty$ is the smallest real number such that 
\begin{equation} \label{exst-Dx}
  \sup_{y \in [L, \, x \vee M]}  \E \big[ \xi(y) \, \ind(\xi(y) > D_x) \big] \leq \Delta_x/2.
\end{equation}  
Such $D_x$ exists by Assumption~\ref{ex-storage-cond1}(iii) and the fact that viewed as a function of $z$, 
$$\sup_{y \in [L, \, x \vee M]}  \E \big[ \xi(y) \, \ind(\xi(y) > z) \big]$$ is continuous from the right. The upper bounds in (\ref{exst-ineq3}) and (\ref{exst-ineq4}) are thus finite (cf.~Assumption~\ref{ex-storage-cond1}(iv)).
  
Combining (\ref{exst-ineq1b})-(\ref{exst-ineq3}), we have
$$ \textstyle{ \E_x^\mu \left[ \, \sum_{n=0}^{\tau -1} c(x_n, a_n) + c(x_{\tau}, y_\alpha) \right]} \leq H(x) \qquad \forall \, x \in \R, \ \alpha \geq \alpha_{\bar \epsilon}, $$
where the function $H$ is finite everywhere, independent of $\alpha$, and given by
\begin{equation} \label{eq-H1}
  H(x) = \kappa(M-x) + \textstyle{\sup_{y \in [L, M]} } \, \psi(y) \qquad \text{for} \ x < L,
\end{equation}  
and for $x \geq L$,
\begin{equation} 
  H(x) = 
 \frac{2 (x - L + D_x)}{\Delta_x} \cdot \left\{ \sup_{y \in [L, \, x \vee M]} \psi(y)  + \sup_{y \in [L, \, x \vee M]} \E \Big[ \kappa \big(M - L + \xi(y)\big) \Big] \right\}. \label{eq-H2}
\end{equation}
We can now apply Lemma~\ref{lem-Sch1} to prove $\sup_{\alpha \in (0,1)} h_\alpha(x) < \infty$ for all $x \in \R$. 

Let $\alpha \geq \alpha_{\bar \epsilon}$. For all $x \leq L$, by (\ref{exst-ya}), $\int_\X v_\alpha(y) \, q(dy \,|\, y_\alpha) \leq \uv_\alpha(x) + {\bar \epsilon}/2$, so for applying Lemma~\ref{lem-Sch1}, we can define a policy $\mu^{{\bar \epsilon}/2}_\alpha$ that satisfies (\ref{eq-mu-al-ep}) for $\epsilon = \bar \epsilon/2$, in such a way that $\mu^{{\bar \epsilon}/2}_\alpha(x) = y_\alpha$ for $x \leq L$.
We have $\alpha \uv_\alpha(x_\tau) \leq m_\alpha + {\bar \epsilon}/2$ by (\ref{ineq-exst-1}), so by Lemma~\ref{lem-Sch1} and the preceding proof, for all $x \in \R$ and $\alpha \geq \alpha_{\bar \epsilon}$,
$$ h_\alpha(x) \leq {\bar \epsilon} +  \textstyle{ \E_x^\mu \left[ \, \sum_{n=0}^{\tau -1} c(x_n, a_n) + c(x_{\tau}, y_\alpha) \right]} \leq {\bar \epsilon} + H(x)  < \infty.$$
As $H(x)$ is independent of $\alpha$, this inequality implies $\limsup_{\alpha \uparrow 1} h_\alpha(x) < \infty$, which, under Assumption~\ref{cond-pc-1}(G), is equivalent to $\sup_{\alpha \in (0,1)} h_\alpha(x) < \infty$ by \cite[Lem.~5]{FKZ12}.
\end{proof}
\smallskip

We now tackle Assumption~\ref{cond-pc-2} (the majorization condition). Let us place two more conditions on the one-stage costs and the demand distributions: 

\begin{assumption}[additional conditions on $c(\cdot)$ and $\{\xi(a)\}$] \label{ex-storage-cond2} \hfill 
\begin{enumerate}[leftmargin=0.65cm,labelwidth=!]
\item[\rm (i)] $\lim_{|a| \to \infty} c(x, a) = +\infty$ for each $x \in \R$.
\item[\rm (ii)] For any bounded interval $I \subset \R$,  $\{F_a \mid a \in I\}$ have densities w.r.t.\ the Lebesgue measure that are uniformly bounded above, and $\cup_{a \in I} \supp(F_a)$, the union of their supports, is bounded. 
\end{enumerate}
\end{assumption}

Assumption~\ref{ex-storage-cond2}(i) will help us eliminate actions and find proper subsets of actions to use as $K_\epsilon(x)$ in the majorization condition. Assumption~\ref{ex-storage-cond2}(ii) and the function $H$ derived in the proof of Proposition~\ref{ex-storage-prp1} will be useful in verifying the uniform integrability condition on $\uh$ in Assumption~\ref{cond-pc-2}(iii), since $\uh \leq H + \bar \epsilon$. In Assumption~\ref{ex-storage-cond2}(ii), the requirement on $\{F_a\}$ to have densities is a limitation of our technique and will be discussed further in Remark~\ref{rmk-maj-density}. The bounded-support condition is much stronger than Assumption~\ref{ex-storage-cond1}(iii); it can be relaxed if one knows more precisely how the tails of $F_a, a \in I$, decay and how $H(x)$ varies with $x$ as $x \to - \infty$.

\begin{prop} \label{prp-expc-2}
Assumptions~\ref{ex-storage-cond1} and~\ref{ex-storage-cond2} imply Assumption~\ref{cond-pc-2}. 
\end{prop}

\begin{proof}
Let $\bar \epsilon > 0, \alpha_{\bar \epsilon} \in (0,1)$, and the function $H$ be as in the proof of Proposition~\ref{ex-storage-prp1}. 
That proof showed that 
under Assumption~\ref{ex-storage-cond1}, $h_\alpha(x) \leq {\bar \epsilon} + H(x)$ for all $x \in \R$ and $\alpha \geq \alpha_{\bar \epsilon}$. Consider now an arbitrary state $x$.
By Assumption~\ref{ex-storage-cond2}(i), there exists $\bar a \geq x$ such that $c(x,a) \geq \underline{g}^* + 2 {\bar \epsilon} + H(x)$ for all $a \geq \bar a$. 
Then, from the $\alpha$-DCOE
$$ \textstyle{(1-\alpha) \, m_\alpha + h_\alpha(x) = \inf_{a \in A(x)} \left\{ c(x,a) + \alpha \int_\X h_\alpha(y) \, q(dy \mid x, a) \right\}}$$
and the fact $\limsup_{\alpha \uparrow 1} (1 - \alpha) \, m_\alpha \leq \underline{g}^*$, we have that for all $\alpha$ sufficiently large,
$$ \textstyle{ \inf_{a \in [x, \, \bar a] } \left\{ c(x,a) + \alpha \int_\X h_\alpha(y) \, q(dy \mid x, a) \right\} = (1-\alpha) \, m_\alpha + h_\alpha(x).}$$
Thus, for any $\epsilon > 0$, Assumption~\ref{cond-pc-2}(i) holds with $K_\epsilon(x) = K : = [x, \bar a]$. 
Next, for the majorizing finite measure $\nu$ required in Assumption~\ref{cond-pc-2}(ii), in view of Assumption~\ref{ex-storage-cond2}(ii), we can simply let $\nu$ be a multiple of the Lebesgue measure on the bounded set $E : = \cup_{a \in K} \, \{ a - \supp(F_a) \}$ and the trivial measure on the complement set $E^c$.

Finally, given Assumption~\ref{ex-storage-cond2}(ii) and the fact $\uh \leq \bar \epsilon+ H$, for Assumption~\ref{cond-pc-2}(iii) to hold, it suffices that $H$ is bounded on bounded intervals. 
The expression (\ref{eq-H1})-(\ref{eq-H2}) of $H$ shows that this is the case. In particular, by Assumption~\ref{ex-storage-cond1}(i) and (iv), if $x$ lies in a bounded interval $I$, all those terms in (\ref{eq-H1})-(\ref{eq-H2}) that involve the functions $\kappa(\cdot)$ and $\psi(\cdot)$ must be bounded above by some constants depending on $I$. The remaining term is $2 (x - L + D_x)/\Delta_x$ (which bounds $\E^\mu_x [ \tau]$) for $x \geq L$. Let $I$ be a bounded interval in $[L, \infty)$. By Assumption~\ref{ex-storage-cond1}(iv), we have 
$$\Delta : = \inf_{x \in I} \Delta_x = \inf_{x \in I} \inf_{y \in [L, \, x \vee M]} \E \big[ \xi(y) \big]  > 0.$$ 
We also have $\sup_{x \in I} D_x < \infty$, because by Assumption~\ref{ex-storage-cond1}(iii), there exists some $0 < D < \infty$ such that 
$$ \textstyle{\sup_{y \geq L}}  \, \E \big[ \xi(y) \, \ind(\xi(y) > D) \big] \leq \Delta/2 \leq \Delta_x/2 \qquad \forall \, x \in I,$$
and this implies that $D \geq D_x$ for all $x \in I$, since $D_x$ is by definition the smallest number satisfying (\ref{exst-Dx}).
Hence $\sup_{x \in I} 2 (x - L + D_x)/\Delta_x < \infty$. Thus $H$ is bounded on any bounded interval, and consequently, Assumption~\ref{cond-pc-2}(iii) holds.
\end{proof}
\smallskip

We have shown that the conditions in Theorem~\ref{thm-pc-acoi} are met and the ACOI holds.
In connection with this example, let us make two final remarks regarding the structure of stationary optimal/$\epsilon$-optimal policies and limitation of our majorization technique. 

\begin{rem} \label{rmk-maj-optpol}
For l.s.c.~models studied in e.g., \cite{FKZ12,FeL18,HL99,JaN06,Sch93}, besides the ACOI or ACOE, one can also relate the optimal actions for discounted problems to those for the average cost problem in the limit as the discount factor $\alpha \uparrow 1$. Moreover, for inventory control, if certain convexity properties are present, there exist optimal policies with particularly simple structures like the policy $\mu$ in this example, i.e., the so-called $(s,S)$ policies and their generalizations (see e.g., \cite{ChS04,Fei16,FeL18}). 
For discontinuous MDP models, in general, one cannot guarantee such structural properties for the stationary optimal or $\epsilon$-optimal policies. The proofs of Theorems~\ref{thm-uc-acoi} and~\ref{thm-pc-acoi} show that if for $\epsilon > 0$, the set-valued map $x \mapsto K_\epsilon(x)$ has an analytic graph, then in the average cost problem, there exists a nonrandomized stationary $\epsilon$-optimal policy $\mu$ such that $\mu(x) \in K_\epsilon(x)$ for all $x \in \X$. This, however, provides only a ``loose'' connection between some good actions in the discounted problems and those in the average cost problem, for the sets $K_\epsilon(x)$ can be quite large, as shown by this example for (PC) and the example for (UC) in Appendix~\ref{appsec-ex-uc}.
\myqed
\end{rem}

\begin{rem} \label{rmk-maj-density}
Typically, the random demands in inventory control problems are allowed to be zero, as long as they are positive with positive probabilities. We do not allow this in Assumption~\ref{ex-storage-cond2}(ii), because no finite measure can majorize the measures $r_a \delta_a$ with $r_a > 0$ for all $a \in I$. This limitation is due to the nature of the measure-theoretic properties underlying our majorization condition, as discussed in Sections~\ref{sec-uc-cond} and~\ref{sec-maj-wsc}.
For similar problems in which $q(dy\,|\,x,a)$ is not atomless, one possible way to proceed is to approximate the original MDP by one in which each action corresponds to some probability distribution over the feasible actions in the original problem. Then the majorization idea and the analysis technique presented in this paper may still be potentially useful for analyzing the approximating MDP, thereby shedding light on the original average cost problem.\myqed
\end{rem}

\begin{appendices}

\section{Supplementary Materials for Section~\ref{sec-5}} \label{appsec-1}

\subsection{Proofs of Proposition~\ref{prp-inf-epilim} and Lemma~\ref{lem-Sch2}} \label{appsec-prf}

\begin{proof}[Proof of Proposition~\ref{prp-inf-epilim}]
We prove the second statement first. By the first inequality in (\ref{ineq-epi1}), we need to prove $\lim_{n \to \infty} \inf_{y \in Y} f_n(y) \geq \inf_{y \in Y} \!\big({\textstyle{\elims_n}} \, f_n \big)(y)$, and in turn, by the assumption (\ref{epi-cond2}), it is sufficient to prove that for each $\epsilon > 0$, $\ell : = \liminf_{j \to \infty} \inf_{y \in K} f_{n_j}(y) \geq \inf_{y \in Y} \big({\textstyle{\elims_n}} \, f_n \big)(y)$. To this end, consider a sequence $\{y_j\}$ in $K$ with $f_{n_j}(y_j) -  \inf_{y \in K} f_{n_j}(y) \to 0$ as $j \to \infty$. 
Since $K$ is compact, there exists a subsequence $y_{j_i} \to \bar y \in K$ such that, with $n'_i : = n_{j_i}$ and $y'_i : = y_{j_i}$, 
$f_{n'_i}(y'_i) \to \ell$ as $i \to \infty$. Since $\big(y'_{i}, f_{n'_i}(y'_{i}) \big) \in \epi (f_{n'_i})$ for all $i$ and $y'_{i} \to \bar y$,
by the definition of the lower epi-limit, we have $\big({\textstyle{\elims_n}} \, f_n \big)(\bar y) \leq \lim_{i \to \infty} f_{n'_i}(y'_{i}) = \ell$ and hence $\inf_{y \in Y} \!\big({\textstyle{\elims_n}} \, f_n \big)(y) \leq \ell$ as desired.

To prove the first statement, for every $\epsilon > 0$, let $\bar y$ be any point that satisfies 
$$ \big({\textstyle{\elims_n}} \, f_n \big)(\bar y) \leq \inf_{y \in Y} \big({\textstyle{\elims_n}} \, f_n \big)(y) + \epsilon/3.$$
By the definition of the lower epi-limit, there exist a subsequence $\{n_j\}$ and points $\{y_j\}$ with 
$y_j \to \bar y$ and $f_{n_j}(y_j) \to  \big({\textstyle{\elims_n}} \, f_n \big)(\bar y)$ as $j \to \infty$. 
Let $K = \{y_j\}_{j\geq1} \cup \{ \bar y\}$, which is a compact subset of $Y$. Since $\inf_{y \in K}  f_{n_j}(y) \leq f_{n_j}(y_j)$, it follows that for all $j$ sufficiently large,
$$  \inf_{y \in K}  f_{n_j}(y) \leq \big({\textstyle{\elims_n}} \, f_n \big)(\bar y) + \epsilon/3 \leq \inf_{y \in Y} \! \big({\textstyle{\elims_n}} \, f_n \big)(y)  + 2\epsilon/3,$$
whereas (\ref{epi-cond}) implies that $\inf_{y \in Y} \!\big({\textstyle{\elims_n}} \, f_n \big)(y) \leq \inf_{y \in Y} f_{n_j}(y) + \epsilon/3$ for all $j$ sufficiently large. Hence (\ref{epi-cond2}) holds, as desired. 
\end{proof}

\begin{proof}[Proof of Lemma~\ref{lem-Sch2}]
First, note that for all $\alpha$ larger than some $\alpha_\epsilon < 1$, we have $(1-\alpha) \um_\alpha \leq \underline{g}^* + \epsilon/4$. This follows from the fact that $\alpha (\um_\alpha - m_\alpha) \leq (1 - \alpha) m_\alpha$ (since $\alpha \, \um_\alpha \leq m_\alpha \leq \um_\alpha$) and $\limsup_{\alpha \to \infty} (1 - \alpha) m_\alpha \leq \underline{g}^*$. 

Consider an arbitrary $\alpha \geq \alpha_\epsilon$ and $(x,a) \in G$. Let $\pi$ be a policy that applies action $a$ at $x_0=x$ and then applies a stationary policy that is $\epsilon/4$-optimal for the $\alpha$-discounted problem. 
Let $\tau = \inf \big\{ n \geq 0 \mid (x_n, a_n) \not\in G \big\}$ ($\tau = \infty$ if this set is empty).
We have
\begin{align*}
\!\! \int_\X \!\! v_\alpha(y) \, q(dy \,|\, x, a)  + \frac{\epsilon}{4} & \geq \E^\pi_x \Big[ \sum_{n=1}^\infty \alpha^{n-1} c(x_n, a_n) \Big] \\
      & \geq \E^\pi_x \Big[ \sum_{n=1}^{\tau} \alpha^{n-1} c(x_n, a_n) + \alpha^{\tau} v_\alpha (x_{\tau+1}) \Big]  \\
      & =   \sum_{n=1}^{\infty}  \E^\pi_x \big[\ind(\tau > n - 1) \, \alpha^{n-1} c(x_n, a_n)  + \ind(\tau = n) \, \alpha^{n} v_\alpha (x_{n+1}) \big].   
\end{align*}
In view of the property of $G$, for $n \geq 1$, on $\{\tau > n-1\}$,
$$  \E^\pi_x \big[ c(x_n, a_n) \mid \F_{n-1} \big] \geq \int_\X \uc(x_n) \, q(dx_n \mid x_{n-1}, a_{n-1} \big)  \geq \underline{g}^* + \epsilon \geq (1-\alpha) \,\um_\alpha + 3\epsilon/4,$$
and in particular, $\E^\pi_x \big[ c(x_1, a_1) \big] \geq (1-\alpha) \,\um_\alpha + 3\epsilon/4$.
We also have $\E^\pi_x \big[ v_\alpha(x_{n+1}) \mid \F_n \big] \geq \um_\alpha$.
Combining these relations, we obtain that
\begin{align*}   
 \!\int_\X v_\alpha(y) \, q(dy \mid x, a)  + \epsilon/4 & \geq 3\epsilon/4 + 
        \E^\pi_x \Big[ \sum_{n=1}^{\tau} \alpha^{n-1} \cdot (1- \alpha) \um_\alpha + \alpha^{\tau} \um_\alpha \Big] 
      = 3\epsilon/4 +   \um_\alpha,
\end{align*}
which gives the desired inequality for all $\alpha \geq \alpha_\epsilon$ and $(x,a) \in G$.
\end{proof}

\subsection{Derivations of (\ref{exst-ineq4})-(\ref{exst-ineq3}) for the Proof of Proposition~\ref{ex-storage-prp1}} \label{appsec-derivations-ex-pc}

We first prove the bound (\ref{exst-ineq3}) on $\E_x^\mu [\tau]$ for $x \geq L$.
Recall $\Delta_x = \inf_{y \in [L, \, x \vee M]} \E \big[ \xi(y) \big] > 0$ and $0 < D_x < \infty$ is such that 
$$  \sup_{y \in [L, \, x \vee M]}  \E \big[ \xi(y) \, \ind(\xi(y) > D_x) \big] \leq \Delta_x/2.$$
Consider nonnegative random variables $Z_n : =  (x_n - L + D_x )^+$ for $n \geq 0$, where $(\cdot)^+ : = \max \{\cdot, 0\}$.
Denote by $\F_n$ the $\sigma$-algebra generated by $(x_0, a_0, \ldots, x_n, a_n)$. Using the definition of the policy $\mu$, a direct calculation shows that when $x_n \geq L$,
$$ \E_x^\mu \big[ Z_{n+1} \mid \F_n \big] = \E_x^\mu \big[ ( x_n - \xi_n(x_n) - L + D_x )^+ \mid \F_n \big] \leq Z_n - \tfrac{1}{2} \Delta_x, $$
and when $x_n < L$, 
$$\E_x^\mu \big[ Z_{n+1} \mid \F_n \big] \leq Z_n + y_\alpha - L + D_x.$$
By a comparison theorem \cite[Prop.\ 11.3.2]{MeT09} (which is an implication of Dynkin's formula~\cite[Thm.~11.3.1]{MeT09}), this implies that for the stopping time $\tau = \inf \{ n \geq 0 \mid x_n < L \}$,
$$ \E_x^\mu \big[ \, \textstyle{\sum_{n=0}^{\tau - 1}} \, \tfrac{1}{2} \Delta_x \, \big] \leq Z_0 = x - L + D_x.$$
Therefore, 
$$ \E_x^\mu [\tau]  \leq 2 (x - L + D_x)/\Delta_x, $$
which is the desired inequality~(\ref{exst-ineq3}).

We now bound $\E_x^\mu \big[ \kappa(y_\alpha - x_{\tau}) \big]$ for $x \geq L$ and derive the inequality (\ref{exst-ineq4}).
Since $\kappa(\cdot)$ is non-decreasing by Assumption~\ref{ex-storage-cond1}(i) and 
$$y_\alpha \leq M, \ \ \ x_{\tau -1} \geq L, \ \ \ x_{\tau} =  x_{\tau-1} - \xi_{\tau-1}(x_{\tau-1})$$ 
by the definition of the stopping time $\tau$ and policy $\mu$, we have
\begin{align*}
 \E_x^\mu \big[ \kappa\big(y_\alpha - x_{\tau} \big) \big] & \leq \E_x^\mu \big[ \kappa \big(M - L + \xi_{\tau-1}(x_{\tau-1}) \big) \big]  \\
 & \leq  \E_x^\mu \left[ \sum_{n=0}^{\tau -1}  \kappa \big(M - L + \xi_{n}(x_{n}) \big) \right] \\
 & =  \E_x^\mu \left[ \sum_{n=0}^\infty \ind(\tau > n) \, \kappa \big(M - L + \xi_{n}(x_{n}) \big) \right]  \\
 & = \E_x^\mu \left[ \sum_{n=0}^\infty \ind(\tau > n) \,  \E_x^\mu \Big[ \kappa \big(M - L + \xi_{n}(x_{n}) \big) \mid \F_n \Big] \right] \\
 & \leq   \E_x^\mu \left[ \sum_{n=0}^\infty \ind(\tau > n)  \sup_{y \in [L, \, x \vee M]} \E \Big[ \kappa \big(M - L + \xi(y)\big) \Big] \right].
\end{align*} 
Therefore,
$$   \E_x^\mu \big[ \kappa (y_\alpha - x_{\tau}) \big]  \leq \E_x^\mu [\tau]  \cdot \sup_{y \in [L, \, x \vee M]} \E \Big[ \kappa \big(M - L + \xi(y)\big) \Big].$$
This is the desired inequality (\ref{exst-ineq4}).

\subsection{An Illustrative Example for the (UC) Case} \label{appsec-ex-uc}

This example is adapted from an inventory-production system example in the book~\cite[Examples~8.6.2,~10.9.3]{HL99}. 
In the original example, the random demands at each stage are i.i.d.\ and for the average cost problem, the production level must always be below the expected demand so as to ensure that the $w$-geometric ergodicity condition of \cite[Thm.~10.3.1]{HL99} holds.
We relax these requirements in our example so that the state transition stochastic kernel need not be continuous and the MDP need not be uniformly $w$-geometrically ergodic. To show that the ACOI holds in this example, most of our efforts will be on making sure that $\sup_{\alpha \in (0,1)} \|h_\alpha \|_w < \infty$.

Regarding notation, for $x, y \in \R$, denote $x \wedge y = \min \{x, y\}$ and $x \vee y = \max \{x, y\}$. Recall also that $(x)^+=\max \{ x, 0\}$.

Let the state and action spaces be $\X = \A = \R_+$. 
The states evolve according to 
$$x_{n+1} = (x_n + z_n - \xi_n)^+,$$ 
where $x_n$ is the stock level and $z_n$ the amount of production at the beginning of the $n$th stage, 
and $\xi_n$ is the random demand during that stage. As in Section~\ref{sec-pc-ex}, we let actions correspond to the stock levels after production:
$$ a_n = x_n + z_n.$$
We assume that given $a_n$, the demand $\xi_n$ depends on the value of $a_n$ and is conditionally independent of the history $(x_0, a_0, \xi_0, x_1, \ldots, \xi_{n-1}, x_n)$. As before, we denote the demand distributions by $F_a$, $a \in \R_+$, we write $\xi_n(a_n)$ for $\xi_n$ to emphasize the dependence on $a_n$, and we use $\xi(a)$ to denote a generic random variable with probability distribution $F_a$ and $\E [ \xi(a) ]$ to denote its expectation w.r.t.\ $F_a$.

Let the one-stage cost function be given by
$$c(x, a) = \kappa(z) + \psi(a) - s \, \E [ a \wedge \xi(a) ]   \qquad \text{with} \ z = a -x.$$ 
The functions $\kappa: \R_+ \to \R_+$ and $\psi: \R_+ \to \R_+$ specify the production cost and the maintenance/holding cost, respectively, and $s \,\E \big[ a \wedge \xi(a) \big]$ is the sales revenue with $s$ being the unit sale price. This definition of $c(\cdot)$ is similar to the one in \cite[Example~8.6.2]{HL99}, except that we do not require $\kappa(\cdot)$ and $\psi(\cdot)$ to be continuous or l.s.c. 

\begin{assumption}[conditions on one-stage costs] \label{ex-uc-cond1} 
On $\R_+$, the production cost $\kappa(\cdot)$ is bounded on bounded intervals, and the maintenance or holding cost $\psi(\cdot)$ is bounded above by some polynomial function. 
\end{assumption}

\subsubsection{Definitions of Action Sets}

We assume that beyond certain stock level $L$, the demand $\xi(a)$ becomes ``saturated'' and follows a fixed probability distribution, no longer affected by $a$. Accordingly, we separate the states into three groups,
$$ \{0\}, \qquad (0, \, L),  \qquad [L, \, + \infty).$$
The state $0$ will be special in this example.
The feasible action sets $A(x)$ are bounded intervals of the form $[x, a_x]$, and we define them for each group of states differently, by placing different constraints on the admissible production levels $z$: 
\begin{itemize}[leftmargin=0.7cm,labelwidth=!]
\item[(i)] For all $x \geq L$, $z \in [0, \theta]$ for some $\theta > 0$ (so $A(x) = [x, x+\theta]$). 
We will let the parameter $\theta$ be smaller than the mean demand, in order to satisfy the (UC) model conditions. The precise definition of $\theta$ will be given below. 
\item[(ii)] For $0 < x < L$, $z \in [0, \theta_x]$ for some $\theta_x > 0$, and 
$\sup_{x \in (0, L)} \theta_x < \infty$. There are no other restrictions on $\theta_x$; in particular, the production level $z$ can exceed the expected demand $\E \big[\xi(x+z)\big]$. (Regarding measurability, it suffices that $\theta_x$ is a Borel measurable function of $x$.)
\item[(iii)] For the state $0$, $z \in [0, \bar a]$ where $\bar a$ will be made large enough so that $A(0) = [0, \bar a]$ contains $A(x)$ for a large subset of states $x$. The precise definition of $A(0)$ will be given below. 
The purpose of such a choice is to ensure $\sup_{\alpha \in (0,1)} \| h_\alpha\|_w < \infty$ in this example; more specifically, it will be used to ensure that $\{h_\alpha\}$ can be bounded below by a multiple of the weight function $w$ of our choice.
\end{itemize}
\medskip

We now proceed to define the weight function $w(\cdot)$, the action set $A(0)$, and the maximum production level $\theta$ allowed at states $x \geq L$. Some conditions on the demand distributions will be needed: 

\begin{assumption}[conditions on $\{\xi(a)\}$] \label{ex-uc-cond2} \hfill
\begin{enumerate}[leftmargin=0.6cm,labelwidth=!]
\item[\rm (i)] For $a \geq L$, $F_a$ is the same, independent of $a$.
\item[\rm (ii)]  $\inf_{a > 0} \E \big[ \xi(a) \big] > 0$.
\item[\rm (iii)] $\big\{\xi(a) \mid a \geq 0 \big\}$ is uniformly integrable. 
\end{enumerate}
\end{assumption}
\smallskip

These conditions will be used in the subsequent analysis to ensure that all the conditions required by Theorem~\ref{thm-uc-acoi} are met, except for the majorization condition. (For the latter, another condition will be added; cf.~Assumption~\ref{ex-uc-cond3}.) 

\medskip
\noindent {\bf Choice of $\theta$ and Parameters in the (UC) Model:} 
Similarly to \cite[Example~10.9.3]{HL99}, for states $x \geq L$, we set the upper limit $\theta$ on the production levels to be such that for $a \geq L$,
$$  0 <  \theta <  \E \big[ \xi(a) \big] $$
(cf.\ Assumption~\ref{ex-uc-cond2}(i)-(ii)), 
and we then have that for some $r > 0$ (and all $a \geq L$),
$$ \lambda : =  \E \Big[ e^{r (\theta - \xi(a) )} \Big] < 1.$$
Let the weight function $w$ be 
$$w(x) = e^{r x}.$$ 
By a direct calculation similar to that given in \cite[pp.~72-73]{HL99},
\begin{equation} \label{ineq-ex-uc-0a}
 \sup_{z \in [0, \theta]} \, \int_\X w(y) \, q(dy \mid x, \, x+z) \leq \lambda \, w(x) + 1 \qquad \forall \, x \geq L.
\end{equation} 
For states $x < L$, since $\cup_{x < L} A(x)$ is bounded, we have
$$b : = \sup_{x < L, \, a \in A(x)} \int_\X w(y) \, q(dy \mid x, a) < \infty. $$
The (UC) model condition (b) is then satisfied, for the preceding constants $\lambda, b$ and weight function $w(\cdot)$.
In view of Assumption~\ref{ex-uc-cond1} and the fact that the admissible production levels are uniformly bounded for all states, the (UC) model condition (a) on the one-stage cost function is also satisfied for some constant $\hat c$.

\medskip
\noindent {\bf Choice of $A(0)$:} For the state $0$, we choose an especially large set $A(0)$ as follows.
For some $\lambda' \in (\lambda, 1)$, let $\tilde L = \max\{ L, \, - \ln (\lambda' - \lambda)/r\}$. Then $(\lambda' - \lambda) \, e^{r \tilde L} \geq 1$ and by (\ref{ineq-ex-uc-0a}),
\begin{equation} \label{ineq-ex-uc-0}
  \sup_{z \in [0, \theta]} \int_\X w(y) \, q(dy \mid x, \, x+z) \leq \lambda' w(x) \qquad \forall \, x \geq \tilde L \geq L.
\end{equation}  
Let $A(0) =[0, \bar a]$ be such that
\begin{equation}
    A(0) \supset A(x) \qquad \forall \, 0 < x \leq \tilde L
\end{equation}    
(e.g., let $\bar a = ( \tilde L + \theta) \vee \sup_{x \in (0, L)} (x + \theta_x) < \infty$). As mentioned and will be shown shortly, the purpose of this choice of $A(0)$ is to make sure that in our subsequent analysis, $h_\alpha(x)$ can be bounded from below independently of $\alpha$.

\subsubsection{Bounding $\{h_\alpha\}$} \label{sec-uc-ex-2}

Set $\bar x = 0$ in the definition of the relative value functions $h_\alpha(x) = v_\alpha(x) - v_\alpha(\bar x)$, $\alpha \in (0,1)$.
We prove below that under the preceding assumptions, $\sup_{\alpha \in (0,1)} \|h_\alpha\|_w < \infty$.

Let $\hat \X : = [0, \tilde L]$. We derive upper/lower bounds on $h_\alpha(x)$, first for $x \in \hat \X$ and then for $x \not\in \hat \X$. To calculate the upper bounds, we consider the policy $\pi$ that makes no production so that $z_n = 0$ and $a_n = x_n$ always, and we bound first the expected hitting time to the state $\bar x = 0$ under $\pi$. Let $\bar \tau := \inf \{ n \geq 0 \mid x_n = 0 \}$.

\medskip
\noindent {\bf Bounding $\E_x^\pi \big[\bar \tau\big]$ for $x \in \hat \X$:} The derivation is similar to that of (\ref{exst-ineq3}) in the example in Section~\ref{sec-pc-ex} (see Appendix~\ref{appsec-derivations-ex-pc}).
Let $\Delta : = \inf_{y \in [0, \tilde L]} \E \big[ \xi(y) \big] > 0$ (cf.\ Assumption~\ref{ex-uc-cond2}(ii)). By Assumption~\ref{ex-uc-cond2}(iii), there exists $0 < D < \infty$ such that 
\begin{equation} 
  \sup_{y \in [0, \, \tilde L]}  \E \big[ \xi(y) \, \ind(\xi(y) > D) \big] \leq \Delta/2. \notag
\end{equation}
For $n \geq 0$, define $\tilde x_0 : = x_0 = x$, $\tilde x_{n+1} : = x_n - \xi_n(x_n)$; then $\tilde x_{n+1} = x_{n+1}$ if $\xi_n(x_n) \leq x_n$.
Let $Z_n : =  ( \tilde x_n + D )^+$ and let $\F_n$ be the $\sigma$-algebra generated by $(x_0, x_1, \ldots, x_n)$.
A direct calculation shows that 
$$ \E^\pi_x \big[ Z_{n+1} \mid \F_n \big] \leq \begin{cases}
  Z_n - \Delta/2 & \text{when} \ x_n > 0; \\
  Z_n + D & \text{when} \ x_n = 0.
  \end{cases} $$
Applying a comparison theorem \cite[Prop.\ 11.3.2]{MeT09} to the nonnegative random variables $\{Z_n\}$, we obtain 
$$ \E^\pi_x \big[ \textstyle{\sum_{n=0}^{\tau - 1}} \tfrac{1}{2} \Delta \big] \leq Z_0 = x + D,   $$
so
\begin{equation} \label{ineq-ex-uc-1}
  \E^\pi_x \big[ \bar \tau \big] \leq 2 (x + D)/\Delta \leq 2 (\tilde L + D)/\Delta.
\end{equation}  

\smallskip
\noindent {\bf Bounding $h_\alpha(x)$ from above for $x \in \hat \X$:} The calculations and reasoning are similar to those given in Example~\ref{ex-pt-invmod}. We have
\begin{align}
 v_\alpha(x) & \leq  \E_{x}^{\pi} \big[ \textstyle{\sum}_{n=0}^{\bar \tau-1} \, \alpha^n c(x_n, a_n)  +  \alpha^{\bar \tau} v_\alpha(\bar x) \big] \notag \\
 & \leq \E_x^{\pi} \big[ \textstyle{\sum}_{n=0}^{\bar \tau-1} \, \alpha^n c(x_n, a_n) \big]  +  v_\alpha(\bar x)  + \E_{x}^{\pi} \big[  1 - \alpha^{\bar \tau}\big] \cdot | v_\alpha(\bar x) | \notag \\
 & \leq \hat c \, \E_{x}^{\pi} \big[ \bar \tau \big] \cdot \textstyle{\sup_{x' \in \hat \X}} \, w(x')  + v_\alpha(\bar x)  + \ell_{\bar x} \, \E_{x}^{\pi} \big[ \bar \tau \big], \label{ineq-ex-uc-2}
\end{align}
where, to derive the last inequality, for the first term, we used the (UC) model condition (a) and the fact that the states $x_n \in \hat \X$ for all $n$ under $\pi$, and for the last term, we used the fact that $\E_{x}^{\pi} \big[  1 - \alpha^{\bar \tau}\big] \leq (1 - \alpha) \, \E_{x}^{\pi} \big[ \bar \tau \big]$ and in (UC), $|( 1 - \alpha) v_\alpha(\bar x)|$ is bounded by some constant $\ell_{\bar x}$ for all $\alpha \in (0,1)$. 

\smallskip
\noindent {\bf Bounding $h_\alpha(x)$ from below for $x \in \hat \X$:} By the construction of this example, $A(\bar x) = A(0) \supset A(x)$ for all $x \in \hat \X = [0, \tilde L]$ and $q(dy \,|\, x, a)$ depends on $a$ only. Therefore,
\begin{equation} \label{ineq-ex-uc-3}
\textstyle{  v_\alpha(x) - v_\alpha(\bar x) \geq - \sup_{a \in A(x)} \big| c(x, a) - c(\bar x, a) \big| \geq - 2 \hat c \, \sup_{x' \in \hat \X} w(x').}
\end{equation}   

\smallskip
\noindent {\bf Bounding $h_\alpha(x)$ for $x \not\in \hat \X$:} Let $\tau = \inf \{ n \geq 0 \mid x_n \in \hat \X \}$. In view of (\ref{ineq-ex-uc-0}), the same derivations (\ref{eq-invmod2})-(\ref{ineq-part-invmod2}) and reasoning given in Example~\ref{ex-pt-invmod} apply here and yield the inequalities:
\begin{align*}
    v_\alpha(x) & \leq   \hat c \, \ell w(x) + v_\alpha(\bar x) + \E_x^{\mu_\alpha} \big[ \alpha^\tau v_\alpha(x_{\tau}) - v_\alpha(\bar x) \big], \\    
      v_\alpha(x) & \geq - \hat c \, \ell w(x) + v_\alpha(\bar x) +  \E_x^{\mu_\alpha} \big[ \alpha^\tau v_\alpha(x_{\tau}) - v_\alpha(\bar x) \big]  - \epsilon,
\end{align*}
where $\ell > 0$ is some constant independent of $\alpha$ and $x$; $\epsilon > 0$ is some arbitrary small number; and $\mu_\alpha$ is a stationary $\epsilon$-optimal policy for the $\alpha$-discounted problem. Similarly to the derivations of (\ref{eq-invmod3})-(\ref{eq-invmod3b}) given in Example~\ref{ex-pt-invmod}, we can bound the third term in the above inequalities as
$$ \big| \E_x^{\mu_\alpha} \big[ \alpha^\tau v_\alpha(x_{\tau}) - v_\alpha(\bar x)  \big]  \big| \leq \textstyle{ \sup_{x' \in \hat \X}} \big| v_\alpha(x') - v_\alpha(\bar x) \big| + \ell \ell_{\bar x} w(x),$$
where, by (\ref{ineq-ex-uc-1})-(\ref{ineq-ex-uc-3}),
$$ \textstyle{ \sup_{x' \in \hat \X} } \, \big| v_\alpha(x') - v_\alpha(\bar x) \big| \leq 2 \hat c \, e^{r \tilde L} + \big( \hat c \, e^{r \tilde L} + \ell_{\bar x} \big) \cdot 2 (\tilde L + D)/\Delta,$$
which is a constant independent of $\alpha$.

Combining the preceding relations shows that for some constants $\tilde \ell_1, \tilde \ell_2$ independent of $\alpha$ and $x$,
$$\big| v_\alpha(x) - v_\alpha(\bar x) \big| \leq \tilde \ell_1 w(x) + \tilde \ell_2, \qquad x \in \X.$$ 
It now follows that $\sup_{\alpha \in (0,1)} \|h_\alpha\|_w < \infty$, so Assumption~\ref{cond-uc-1} is satisfied in this example.

\subsubsection{Satisfying the Majorization Condition}

We impose an additional condition on the demand distributions:

\begin{assumption} \label{ex-uc-cond3} 
The demand distributions $\{F_a \mid a > 0 \}$ have densities w.r.t.\ the Lebesgue measure that are uniformly bounded above.
\end{assumption}

The majorization condition (Assumption~\ref{cond-uc-2}) holds easily under the preceding assumptions. In particular, for each state $x \geq 0$ and $\epsilon > 0$, with $A(x) = [x, a_x]$, we observe the following:
\begin{itemize}[leftmargin=0.5cm,labelwidth=!]
\item Let $K_\epsilon(x) = A(x)$; then Assumption~\ref{cond-uc-2}(i) holds trivially. 
\item For Assumption~\ref{cond-uc-2}(ii), the required majorizing finite measure $\nu$ can be defined as 
$$\nu : = f^{\text{max}} \varphi + \delta_{0},$$ 
where $f^{\text{max}}$ is an upper bound on the densities of $\{F_a \mid a > 0 \}$ (cf.~Assumption~\ref{ex-uc-cond3}); $\varphi$ equals the Lebesgue measure on $[0, a_x]$ and the trivial measure on $(a_x, \infty)$; and $\delta_{0}$ is the Dirac measure at $0$. This finite measure $\nu$ majorizes $q(dy \,|\, x, a)$, $a \in [x, a_x]$ and thus meets the requirement in  Assumption~\ref{cond-uc-2}(ii).
\item Finally, notice that $x_{1} \in [0, a_x]$ if $x_0=x$, whereas $w(y) = e^{r y}$ is bounded on $[0, a_x]$. Therefore, the uniform integrability condition required by Assumption~\ref{cond-uc-2}(iii) is also trivially satisfied.
\end{itemize}
\smallskip

We have now verified all the conditions in Theorem~\ref{thm-uc-acoi} and can therefore conclude that the ACOI holds in this example.

\end{appendices}

\section*{Acknowledgments}
The author would like to thank Professor Eugene Feinberg and two anonymous reviewers for important critical comments on the previous version of this manuscript, and for pointing her to several related early and recent works on Borel-space MDPs. The author is also grateful to Dr.~Martha Steenstrup, who read parts of her preliminary draft and gave her advice on improving the presentation.

\addcontentsline{toc}{section}{References} 
\bibliographystyle{siamplain} 
\let\oldbibliography\thebibliography
\renewcommand{\thebibliography}[1]{%
  \oldbibliography{#1}%
  \setlength{\itemsep}{0pt}%
}
{\fontsize{9}{11} \selectfont
\bibliography{ac_BorelDP_bib_new}}


\end{document}